\documentclass[a4paper,11pt]{amsart}
\usepackage{amsmath}
\usepackage{amsthm,amsfonts,amssymb,graphics}
\usepackage[foot]{amsaddr}
\usepackage{verbatim}
\usepackage{hyperref}
\usepackage{xcolor}
\usepackage[normalem]{ulem}
\usepackage{enumerate}
\usepackage{tikz}
\usepackage{geometry}
\usepackage{dsfont}
\usepackage{centernot}
\usepackage{bbm}

\newtheorem{theorem}{Theorem}[section]
\newtheorem{lemma}[theorem]{Lemma}
\newtheorem{proposition}[theorem]{Proposition}
\newtheorem{corollary}[theorem]{Corollary}

\newtheorem{example}[theorem]{Example}

\newtheorem{question}[theorem]{Question}
\theoremstyle{remark}
\newtheorem{remark}[theorem]{Remark}

\numberwithin{equation}{section}

\newcommand \id{\mathds 1}
\newcommand {\R} {\mathbb{R}}
\newcommand {\E} {\mathbb{E}}
\newcommand {\N} {\mathbb{N}}
\newcommand {\Z} {\mathbb{Z}}
\renewcommand{\P} {\mathbb{P}}
\newcommand {\Var} {\mathrm{Var}}

\newcommand {\Cov} {\mathrm{Cov}}
\newcommand {\Capa} {\mathrm{Cap}}

\newcommand {\eps}{\varepsilon}

\begin{document}
\title[A Gaussian sprinkled decoupling inequality and applications]{A sprinkled decoupling inequality for \\ Gaussian processes and applications} 
\author{Stephen Muirhead}
\address{School of Mathematics and Statistics, University of Melbourne}
\email{smui@unimelb.edu.au}
\begin{abstract}
We establish the sprinkled decoupling inequality
\[  \P[X \in A_1 \cap A_2] - \P[X + \eps  \in A_1]\P[X + \eps  \in A_2] \le \frac{c \|K_{I_1,I_2}\|_\infty}{ \eps^2} , \]
where $X$ is an arbitrary Gaussian vector, $A_1$ and $A_2$ are increasing events that depend on coordinates $I_1$ and $I_2$ respectively, $\eps > 0$ is a sprinkling parameter, $\|K_{I_1,I_2}\|_\infty$ is the maximum absolute covariance between coordinates of $X$ in $I_1$ and $I_2$, and $c > 0$ is a universal constant. As an application we prove the non-triviality of the percolation phase transition for Gaussian fields on $\Z^d$ or $\R^d$ with (i) uniformly bounded local suprema, and (ii) correlations which decay at least polylogarithmically in the distance with exponent $\gamma > 3$; this expands the scope of existing results on non-triviality of the phase transition, covering new examples such as non-stationary fields and monochromatic random waves.
\end{abstract}
\date{\today}
\thanks{The author is supported by the Australian Research Council (ARC) Discovery Early Career Researcher Award DE200101467, and acknowledges the hospitality of the Statistical Laboratory, University of Cambridge, where part of this work was carried out. We also thank Michael McAuley, Alejandro Rivera, Pierre-Fran\c{c}ois Rodriguez, and Hugo Vanneuville for interesting discussions on this topic, and an anonymous referee for helpful suggestions.}

\maketitle

\section{Sprinkled decoupling inequalities}

In this paper we study decoupling inequalities of the form
\begin{equation}
\label{e:sdi}
  \P[X \in A_1 \cap A_2] - \P[X + \eps \in A_1]\P[X + \eps \in A_2]  \le \text{`small error',}
\end{equation}
where $X$ is a random vector or process, $A_1$ and $A_2$ are increasing events, and $\eps$ is a small `sprinkling' parameter; such `sprinkled decoupling inequalities' play a key role in the percolation theory of strongly-correlated systems (e.g.\ strongly-correlated Gaussian models \cite{rs13, pr15, cn21, ms22}, Poissonian models such as random interlacements \cite{sni12, pt15, dpr18}, random walk loop soups \cite{as19} and the cylinder model \cite{tw12,at21}, and gradient Gibbs measures \cite{rod16}). To explain the terminology of `sprinkling', consider the case that $X$ is an i.i.d.\ Gaussian vector and $A_1,A_2$ are events that depend only on the excursion set $\{X \ge u\}$ for some $u \in \R$. Since $\id_{\{X_i \ge u\}}$ has the law of a Bernoulli process~$\xi$, the addition of a small $\eps > 0$ in the second term in \eqref{e:sdi} is equivalent to superimposing (i.e.\ `sprinkling') $\xi$ with an independent Bernoulli process of small parameter. 

\smallskip
Note that the presence of `sprinkling' weakens the inequality compared to a non-sprinkled decoupling inequality of the form
\begin{equation}
\label{e:nsdi}
  \big| \P[X \in A_1 \cap A_2] - \P[X  \in A_1]\P[X  \in A_2]\big|  \le \text{`small error'}.
\end{equation}
Nevertheless, when working in off-critical regimes, one can usually tolerate the presence of sprinkling if it is arbitrarily small, and in multi-scale arguments, if it is summable over the scales (see Section~\ref{s:per} for an example).

\smallskip
In this paper we establish a general sprinkled decoupling inequality for Gaussian processes, discrete or continuous. In Section \ref{s:per} we present an application in Gaussian percolation theory, and in Section \ref{s:nsd} we discuss consequences for non-sprinkled decoupling.

\subsection{A sprinkled decoupling inequality for Gaussian processes}
Let $X = (X_i)_{1 \le i \le n}$ be a Gaussian vector with covariance kernel $K(i,j) = \textrm{Cov}[X_i, X_j]$. For $\eps > 0$, we write $X + \eps$ to denote $X + \eps \id$, where $\id$ is the vector of ones. For $I,J \subseteq \{1, \ldots,n \}$, let $K_{I,J} =  ( K(i,j) )_{i \in I, j \in J}$. 

\smallskip
An event $A$ is \textit{increasing} if $\{X \in A\} \subseteq \{X + v \in A\}$ for every $v  \in \R^n$ such that $v \ge 0$, and is \textit{supported on $I \subseteq \{1, \ldots, n\}$}, denoted $A \in \sigma(I)$, if $\{X \in A\} = \{X + v \in A\}$ for every $v \in \R^n$ such that $v|_I = 0$.

\smallskip
Our main result is the following:

\begin{theorem}
\label{t:sdi}
There exists a universal constant $c > 0$ such that, for all $I_1,I_2 \subseteq \{1,\ldots,n\}$, increasing events $A_1 \in \sigma(I_1)$ and $A_2 \in \sigma(I_2)$, and $\eps > 0$,
\begin{equation}
\label{e:tsdi1}
 \P[X \in A_1 \cap A_2] - \P[X + \eps \in A_1]\P[X + \eps \in A_2] \le \frac{c \|K_{I_1,I_2}\|_\infty}{ \eps^2},
 \end{equation}
 and
 \begin{equation}
\label{e:tsdi2}
  \P[X - \eps \in A_1]\P[X - \eps \in A_2]  - \P[X \in A_1 \cap A_2]  \le \frac{c \|K_{I_1,I_2}\|_\infty  }{\eps^2}.
 \end{equation}
If $K_{I_1,I_2} \ge 0$ then \eqref{e:tsdi1} holds with $c=1$ and \eqref{e:tsdi2} holds with $c = 0$.
\end{theorem}

\begin{remark}
As explained above, our main interest is \eqref{e:tsdi1}, although we use \eqref{e:tsdi2} in Section \ref{s:nsd} to obtain \textit{two-sided} bounds in non-sprinkled decoupling inequalities. In the discussion below we focus on \eqref{e:tsdi1}, but most remarks apply to \eqref{e:tsdi2} after relevant notational changes.
\end{remark}

\begin{remark}
\label{r:inhom}
By replacing $X$ with $X' = (X|_{I_1}, X|_{I_2})$, without loss of generality one can assume in Theorem \ref{t:sdi} that $I_1$ and $I_2$ are disjoint. Then by rescaling $X$ one can extend \eqref{e:tsdi1} to an \textit{inhomogeneous} sprinkled decoupling inequality
\[  \P[X \in A_1 \cap A_2] - \P[X + \tilde{\eps}_1 \in A_1]\P[X + \tilde{\eps}_2 \in A_2]   \le   c \max_{i \in I_1, j \in I_2} \bigg| \frac{ K(i,j) }{(\tilde{\eps}_1)_i (\tilde{\eps}_2)_j } \bigg| \]
 for arbitrary sprinkling vectors $\tilde{\eps}_1,\tilde{\eps}_2 > 0$. In particular, for $\eps_1, \eps_2 > 0$,  
 \[  \P[X \in A_1 \cap A_2] - \P[X + \eps_1 \in A_1]\P[X + \eps_2 \in A_2]   \le   \frac{c \|K_{I_1,I_2}\|_\infty}{ \eps_1 \eps_2} .\]
\end{remark}

\begin{remark}
A notable feature of \eqref{e:tsdi1} is that it depends on $K$ only through the \textit{maximum pointwise correlation} $\|K_{I_1,I_2}\|_\infty$. There are various alternative ways to quantify the `correlation' between $X|_{I_1}$ and $X|_{I_2}$, but $\|K_{I_1,I_2}\|_\infty$ is advantageous since (i) it is usually simple to estimate, and (ii) it can be much smaller than other measures. The former is particularly important when dealing with oscillating correlations, and the latter is especially advantageous in the `strongly-correlated' setting in which correlations decay slowly away from the diagonal. In Section \ref{s:dis} below we discuss an alternative sprinkled decoupling inequality which depends on the \textit{maximum correlation coefficient} $\rho(I_1,I_2)$.
\end{remark}

Notice that Theorem \ref{t:sdi} is \textit{dimension free}. As a consequence, using standard approximation arguments one can extend it to continuous Gaussian processes.

\smallskip Let $f$ be a continuous Gaussian process on a domain $D \subseteq \R^d$ with covariance kernel $K(x,y) = \textrm{Cov}[f(x),f(y)]$. For $D_1,D_2 \subseteq D$, let $K_{D_1,D_2} = (K(x,y))_{x \in D_1, y \in D_2}$. An event $A$ is \textit{increasing} if $\{f \in A\} \subseteq \{f + v \in A\}$ for every continuous $v: D \to \R$ such that $v \ge 0$, and is \textit{supported on $D' \subseteq D$}, denoted $A \in \sigma(D')$, if $\{f \in A\} = \{f + v \in A\}$ for every continuous $v: D \to \R$ such that $v|_{D'} = 0$.

\begin{theorem}
There exists a universal constant $c > 0$ such that, for all compact domains $D_1,D_2 \subseteq D$, increasing events $A_1 \in \sigma(D_1)$ and $A_2 \in \sigma(D_2)$, and $\eps  > 0$,
\begin{equation}
\label{e:tcsdi1}
 \P[f \in A_1 \cap A_2] - \P[f + \eps \in A_1]\P[f + \eps \in A_2]  \le \frac{c \|K_{D_1,D_2}\|_\infty  }{\eps^2},
 \end{equation}
 and
 \begin{equation}
\label{e:tcsdi2}
\P[f - \eps \in A_1]\P[f - \eps \in A_2]  -  \P[f \in A_1 \cap A_2]  \le \frac{c \|K_{D_1,D_2}\|_\infty  }{\eps^2}.
 \end{equation}
If $K_{D_1,D_2} \ge 0$ then \eqref{e:tcsdi1} holds with $c = 1$ and \eqref{e:tcsdi2} holds with $c = 0$.
\end{theorem}

\begin{proof}
This is an immediate consequence of Theorem \ref{t:sdi} combined with the following observation: for every continuous random field on $D \subseteq \R^d$, compact $D' \subseteq D$, increasing event $A \in \sigma(D')$, and $\delta > 0$, there exists a finite set $P \subset D'$ and an increasing event $A' \in \sigma(P)$ such that $\P[A  \, \triangle \, A'] \le \delta$ (see \cite[Appendix A]{drrv21}).
\end{proof}

\subsection{Discussion and comparison with related inequalities}
\label{s:dis}

Let us first remark on the optimality of the error in \eqref{e:tsdi1} (i.e.\ the term on the right-hand side). Via rescaling, one can see that if the error depends only on $\|K_{I_1,I_2}\|_\infty$ and $\eps$, then it must do so through  $E =   \eps  / \sqrt{ \|K_{I_1,I_2}\|_\infty}$. The error in \eqref{e:tsdi1} decays \textit{quadratically} in $E$, but it is plausible that one could upgrade this to \textit{Gaussian decay} in general. Such an improvement would have many applications in Gaussian percolation theory, see e.g.\ \cite{drs14,pr15,sap17,ap21,at21}. 

\begin{question}
\label{q:gtb}
Can one replace the error $cE^{-2}$ in \eqref{e:tsdi1} with $c_1 e^{-c_2 E^2}$ for universal $c_1,c_2 > 0$?
\end{question}

An analysis of the bivariate case shows that one cannot hope for error decaying any faster than Gaussian in general:

\begin{proposition}
\label{p:neg}
Suppose there exist $c_1,c_2 > 0$ such that, for all Gaussian vectors $X$, $I_1,I_2 \subseteq \{1,\ldots,n\}$, increasing events $A_1 \in \sigma(I_1)$ and $A_2 \in \sigma(I_2)$, and $\eps > 0$,
\begin{equation}
\label{e:sdineg}
 \P[X \in A_1 \cap A_2] - \P[X + \eps \in A_1]\P[X + \eps \in A_2] \le c_1 e^{- c_2 \eps^2 / \|K_{I_1,I_2}\|_\infty} .
 \end{equation}
 Then $c_2 \le \frac{1}{3-2\sqrt{2}} \approx 5.828\ldots$
\end{proposition}

\smallskip
In a different direction, one can obtain alternative sprinkled decoupling inequalities with Gaussian (or even faster) decay by either (i) replacing $\|K_{I_1,I_2}\|_\infty$ with a different quantifier of correlation, or (ii) restricting the generality of the set-up. We discuss some examples now:

\subsubsection{Maximum correlation coefficient}
For $I_1,I_2 \subseteq \{1,\ldots,n\}$, define the \textit{maximum correlation coefficient} (also called the \textit{Hirschfeld--Gebelein--R\'{e}nyi correlation coefficient})
\[   \rho(I_1,I_2) = \sup_{f \in L_2(X|_{I_1}), g \in L_2(X|_{I_2}) }   \frac{ | \Cov[f(X|_{I_1}), g(X|_{I_2} )] | }{ \sqrt{ \Var[f(X|_{I_1})] \Var[g(X|_{I_2})] } } ,\]
with the convention $0/0 := 0$, and define $\rho(D_1,D_2)$ analogously in the continuous case. For Gaussian vectors, it is a classical fact (see \cite[Theorem 10.11]{jan97}) that $ \rho(I_1,I_2)$ coincides with its linearisation
\begin{equation}
\label{e:rholin}
 \sup_{ \alpha \in \R^{|I_1|} , \beta \in \R^{|I_2|} }   \frac{ | \Cov[  \langle \alpha , X_{I_1} \rangle , \langle \beta , X_{I_2} \rangle ] | }{ \sqrt{ \Var[\langle \alpha , X_{I_1}  \rangle ] \Var[ \langle \beta , X_{I_2}  \rangle] } } .
 \end{equation}
 Clearly $\rho(I_1,I_2)$ satisfies 
\[ 1 \ge \rho(I_1,I_2)  \ge \max_{i \in I_1, j \in I_2} \Big|   \frac{K(i,j)}{\sqrt{ K(i,i) K(j,j)}  } \Big| \ge  \frac{  \|K_{I_1,I_2}\|_\infty }{   \|K\|_\infty  } . \]
However $\rho(I_1,I_2)$ can be much larger than $\|K_{I_1,I_2}\|_\infty /   \|K\|_\infty $, for instance if the pointwise correlations in $X|_{I_1}$ and $X|_{I_2}$ are roughly of the same order, or for Gaussian processes which are real-analytic.
 
 \begin{example}[Gaussian free field]
\label{e:gff}
Suppose $X$ is the Gaussian free field (GFF) on $\Z^d$, $d \ge 3$, i.e.\ the centred stationary Gaussian field with covariance $K(0, x) = G_d(x) \sim c_d \|x\|_2^{-(d-2)}$, where $G_d$ is the Green's function of the simple random walk on $\Z^d$. Fix $k > 2$, and let $I_1$ and $I_2$ be translations of the Euclidean ball $B(R)$ of radius $R \ge 1$ restricted to $\Z^d$, with centres $k R$ apart. Then, as $R \to \infty$, $\rho(I_1,I_2)$ is bounded away from zero (see \eqref{e:caplower}) whereas $\|K_{I_1,I_2}\|_\infty \sim c_k R^{-(d-2)}$.
\end{example}

 \begin{example}[Real analytic fields]
\label{e:cauchy}
Suppose $f$ is a real-analytic Gaussian field on $\R^d$ and let $D_1, D_2 \subset \R^d$ contain open sets. Then $f|_{D_2}$ is a measurable function of $f|_{D_1}$, and so $\rho(D_1,D_2)  = 1$. 
\end{example}

Using ideas from Gaussian isoperimetry, we establish the following:

\begin{theorem}
\label{t:sdi2}
For all $I_1,I_2 \subseteq \{1,\ldots,n\}$, increasing events $A_1 \in \sigma(I_1)$ and $A_2 \in \sigma(I_2)$, and $\eps > 0$,
\begin{equation}
\label{e:tsdi21}
 \P[X \in A_1 \cap A_2] - \P[X  \in A_1]\P[X + \eps \in A_2] \le  \exp \Big(- \frac{ \eps^2 }{8 \|K\|_\infty  \rho^2(I_1,I_2) } \Big). 
 \end{equation}
\end{theorem}

Let us briefly compare \eqref{e:tsdi1} and \eqref{e:tsdi21}. For simplicity suppose $\|K\|_\infty = 1$. Then \eqref{e:tsdi21} has a Gaussian tail in $E' = \eps / \rho(I_1,I_2)$; in particular it decays if $\eps \gg \rho(I_1,I_2)$, whereas \eqref{e:tsdi1} decays if $\eps \gg \sqrt{\|K_{I_1,I_2}\|_\infty }$. Hence  \eqref{e:tsdi21} strictly improves on \eqref{e:tsdi1} if $ \rho^2(I_1,I_2) \le \|K_{I_1,I_2}\|_\infty $. While this may be true in some cases, it is typically not true in strongly-correlated settings (e.g.\ the GFF). Note also that the `sprinkling' in \eqref{e:tsdi21} is only on one domain $I_2$, rather than both.

\begin{remark}
The Gaussian tail in $E'$ is best possible: as in Proposition \ref{p:neg}, any error bound of the form $c_1 \exp ( \frac{-c_2 \eps^2 }{ \|K\|_\infty  \rho^2(I_1,I_2) } )$ must have $c_2 \le \frac{1}{3-2\sqrt{2}} \approx 5.828\ldots$
\end{remark}

\subsubsection{Finite-range approximations}
A common method to analyse dependent Gaussian processes is to approximate them by a \textit{finite-range dependent} version (see, e.g., \cite{cuz76,nsv08}), and in some cases this technique can be used to obtain a sprinkled decoupling inequality with Gaussian error \cite{pr15, mv20, cn21, ms22}. To illustrate the method in a general setting, suppose that for disjoint $I_1,I_2 \subset \{1,\ldots,n\}$ one has a decomposition 
\begin{equation}
\label{e:decomp}
 X \stackrel{d}{=}  X_1 +  X_2, 
 \end{equation}
 where $X_1$ is Gaussian vector such that $X_1|_{I_1}$ and $X_1|_{I_2}$ are independent, and $X_2$ is a centred Gaussian vector not necessarily independent of $X_1$. Then it is straightforward to prove the following inequality, which generalises bounds appearing in \cite{pr15, mv20, cn21, ms22}:

\begin{proposition}
\label{p:ssdi}
For all increasing events $A_1 \in \sigma(I_1)$ and $A_2 \in \sigma(I_2)$, and $\eps > 0$,
\begin{equation}
\label{e:ssdi}
  \P[X \in A_1 \cap A_2] -   \P[X + \eps \in A_1]\P[X + \eps \in A_2] \le 3 \max\{  |I_1|, |I_2| \} e^{-  \eps^2 / (8 \sigma^2 )}   ,
  \end{equation}
  where $\sigma^2 = \max_{i \in I_1 \cup I_2} \Var[ X_2(i)]$.
  \end{proposition}
  
An analogous result holds for continuous processes, except one should replace $ \max\{  |I_1|, |I_2| \} $ with $ \max\{  n_1, n_2 \}$,  where $\mathcal{B}_i  = (B^i_j)_{1 \le j \le n_i}$ are coverings of $D_i$ by translations of the unit ball, and replace $\eps$ on the right-hand side of \eqref{e:ssdi} with $(\eps-2\mu)_+$, where $\mu =\sup_{B \in  \mathcal{B}_1 \cup  \mathcal{B}_2} \E[ \sup_B X_2]$.
  
  \smallskip
The decomposition \eqref{e:decomp} exists in many natural settings, e.g.\ for stationary fields on $T \in \{\Z^d ,\R^d\}$ with `moving average' representation $f = q \star W$, where $q \in L^2(T)$, $W$ is the white noise on $T$ (interpreted as a collection of i.i.d.\ Gaussians if $T = \Z^d$), and $\star$ denotes convolution. Important examples include the GFF on $\Z^d$, $d \ge 3$ \cite{dgrs20, m22} and the Bargmann-Fock \cite{mv20} and Cauchy fields \cite{ms22} on $\R^d$ (the centred isotropic Gaussian fields with respective covariance $K(0,x) = e^{-\|x\|_2^2/2}$ and $K(0, x) = (1 + \|x\|^2_2)^{-\alpha/2}$, $\alpha > 0$). In all these examples one can construct a decomposition with $\sigma^2   \le c \|K_{I_1,I_2}\|_\infty$ for some $c > 0$ that does not depend on $I_1,I_2$. In that case \eqref{e:ssdi} gives 
\begin{equation}
\label{e:ssdi2}
  \P[X \in A_1 \cap A_2] -   \P[X + \eps \in A_1]\P[X + \eps \in A_2] \le 3 \max\{  |I_1|, |I_2| \} e^{-  \eps^2 / (8 c  \|K_{I_1,I_2}\|_\infty) }  ,
 \end{equation}
   This achieves a Gaussian tail bound in $E =  \eps/\sqrt{ \|K_{I_1,I_2}\|_\infty}$ up to linear factors in the size of the domains $I_1$ and $I_2$. As in Question \ref{q:gtb}, it is plausible that this is true in full generality, perhaps even without the linear factors.

\subsubsection{Errorless sprinkled decoupling}
Recently Severo \cite{sev21} showed that, by working within a restricted class of increasing events, in some cases one can prove an \textit{errorless} sprinkled decoupling inequality
\begin{equation}
\label{e:ssdi3}
  \P[X \in A_1 \cap A_2] -   \P[X + \eps \in A_1]\P[X + \eps \in A_2] \le 0.
 \end{equation}
In particular, for both the Bargmann-Fock and Cauchy fields with $\alpha > d$, Severo proved a stochastic domination property that implies that, for every $\eps > 0$, there exists a $R = R(\eps) > 0$ such that \eqref{e:ssdi3} holds for all `crossing events' $A_1$ and $A_2$ on domains $I_1$ and $I_2$ separated by distance~$R$ (see Section \ref{s:per} for examples of these events). More quantitatively, the argument showed that roughly one needs $\eps \ge c \sqrt{\|K_{I_1,I_2} \|_\infty}$ for \eqref{e:ssdi3} to hold.
 
 \smallskip
It would be of interest to understand this phenomenon in more generality:

\begin{question}
Fix $c > 0$. For which Gaussian vectors $X$, $I_1,I_2 \subseteq \{1,2,\ldots,n\}$, and increasing events $A_1 \in \sigma(I_1)$ and $A_2 \in \sigma(I_2)$, does \eqref{e:ssdi3} hold for $\eps = c \sqrt{\|K_{I_1,I_2}\|_\infty} $ (and hence for all $\eps \ge c \sqrt{\|K_{I_1,I_2}\|_\infty} $)?
\end{question}

Proposition \ref{p:neg} shows that \eqref{e:ssdi3} cannot be true in full generality for $\eps = c \sqrt{\|K_{I_1,I_2}\|_\infty}$. This suggests that one must either look, as in \cite{sev21}, to restricted classes of vectors/events, or else replace $\|K_{I_1,I_2}\|_\infty$ with another measure of correlation.

\smallskip
As a step towards the latter, and mirroring Theorem \ref{t:sdi2}, we present a general errorless sprinkled decoupling inequality in which $\sqrt{\|K_{I_1,I_2}\|_\infty}$ is replaced by $\rho(I_1,I_2)$:
 
\begin{theorem}
\label{t:sdi3}
Fix $\delta_1,\delta_2 \in (0,1)$. Then for all $I_1,I_2 \subseteq \{1,\ldots,n\}$ such that $\rho(I_1,I_2) \le 1-\delta_1$, and increasing events $A_1 \in \sigma(I_1)$ and $A_2 \in \sigma(I_2)$ such that $\max_{i=1,2}\P[X \in A_i] \ge \delta_2$, 
\[ \P[X \in A_1 \cap A_2] - \P[X + \eps \in A_1]\P[X + \eps \in A_2] \le 0, \]
 where 
 \[\eps =  \kappa \|K\|_\infty \rho(I_1,I_2) \ , \quad \kappa =2 +   (\delta_1)^{-1/2}  \max \big\{0 , -    \Phi^{-1}(\delta_2)  \big\} ,\]
 and $\Phi$ is the standard Gaussian cdf.
\end{theorem}

Although it is fully general, Theorem \ref{t:sdi3} has two notable disadvantages. First, as discussed above, the dependence on $\rho(I_1,I_2)$ instead of $\|K_{I_1,I_2}\|_\infty$ limits its practical use in some settings. Second, the dependence of $\kappa$ on $\delta_2$ is quite restrictive in applications (e.g.\ in Section \ref{s:per}), when one usually wishes to consider events of small probability. However, since the Gaussian decay in Theorem \ref{t:sdi2} is optimal, some version of this restriction is necessary.

\bigskip

\section{Proof of the sprinkled decoupling inequalities}

In this section we prove our main sprinkled decoupling inequality (Theorem \ref{t:sdi}), and also prove the alternative inequalities presented in Section \ref{s:dis} above, namely Theorems \ref{t:sdi2} and \ref{t:sdi3} and Proposition \ref{p:ssdi}. Finally, we establish the negative result in Proposition \ref{p:neg}.

\subsection{Proof of the Theorem \ref{t:sdi}}
The main ingredient in the proof of Theorem \ref{t:sdi} is an estimate of the covariance between \textit{thresholds} associated to increasing events. We begin by introducing this notion.

\subsubsection{Thresholds for increasing events}

Recall that $X = (X_i)_{1 \le i \le n}$ is a Gaussian vector with covariance $K$. The \textit{threshold} associated to an increasing event $A$ is the random variable
\[ T_A = T_A(X) = \sup\big\{  u \in \R :     \{X - u \in A \} \text{ holds} \big\}  . \]
It satisfies the following basic properties:

\begin{lemma}
\label{l:pt}
Let $I \subseteq \{1,\ldots,n\}$, $A \in \sigma(I)$ be increasing, and assume that $\P[A] \in (0,1)$. Then:
\begin{enumerate}
\item $T_A(X) \in L^2$.
\item $T_A(X)$ is $1-$Lipschitz, and almost surely its gradient $\nabla T_A(X)$ satisfies:
\begin{enumerate}
\item $\frac{ \partial T_A(X)}{\partial X_i}  = 0$ for all $i \notin I$;
\item $\frac{ \partial T_A(X)}{\partial X_i}  \ge 0 $ for all $i \in I$;
 \item $\| \nabla T_A(X) \|_1 = 1 $.
 \end{enumerate}
\item If $X$ is non-degenerate, for every $u \in \R$, $\{X + u \in A \} = \{T_A(X) \le u\}$ almost surely.
\end{enumerate}
\end{lemma}
\begin{proof}
We prove these in turn:\\

 \vspace{-0.3cm}
 \noindent \textbf{(1).} Fix $v^-, v+ \in \R^n$ such that $v^- \notin A$ and $v^+ \in A$ (recall that we assume $\P[A] \in (0,1)$), and set $u^- = \min_i v^-_i$ and $u^+ = \max_i v^+_i$. Then since $A$ is increasing, $ X + u^+ - \min_i X_i \in A$, which implies that  $T_A(X) \ge   \min_i X_i  - u^+$. Similarly $T_A(X) \le   \max_i X_i  + u^-$. Since $\min_i X_i, \max_i X_i$ are square-integrable, so is $T_A(X)$. \\

 \vspace{-0.3cm}
\noindent \textbf{(2).} By the definition of $T_A(X)$, and since $A \in \sigma(I)$ is increasing:
\begin{enumerate}
\item For all $v \in \R^n$, $T_A(X + v) \le T_A(X) + \|v\|_\infty$;
\item  For all $v \in \R^n$ such that $v \ge 0$, $T_A(X+v) \ge T_A(X)$;
 \item For all $v \in \R^n$ such that $v|_I = 0$, $T_A(X+v) = T_A(X)$;
 \item For all $h \in \R$, $T_A(X + h) = T_A(X) + h$.
 \end{enumerate}
 Combining these proves the claim. \\ 
 
 \vspace{-0.3cm}
\noindent \textbf{(3).} By the definition of $T_A(X)$ it suffices to show that $\{T_A(X) = u\}$ has probability zero. For this, observe that since $X$ is non-degenerate, the laws of $X+h$ and $X$ are mutually absolutely continuous for every $h \in \R$. Since also $T_A(X + h) = T_A(X)  + h$, and absolute continuity is preserved under measurable transformation, the laws of $T_A +h$ and $T_A$ are also mutually absolutely continuous for every $h \in \R$. This rules out the existence of atoms in the law of $T_A$.
\end{proof}

The advantage of thresholds in our context is that, assuming cross-correlations are of consistent sign, one can estimate the covariance between thresholds rather precisely:
 
\begin{proposition}
\label{p:dt}
For all $I_1,I_2 \subseteq \{1,\ldots,n\}$, and increasing events $A_1 \in \sigma(I_1)$ and $A_2 \in \sigma(I_2)$ such that $\P[A_i] \in (0,1)$, 
\begin{equation}
\label{e:dt1}
 K_{I_1,I_2} \ge 0 \quad \implies \quad \min_{i \in I_1, j \in I_2} K(i,j)  \le  \textrm{Cov}[ T_{A_1}, T_{A_2} ] \le  \|K_{I_1,I_2} \|_\infty   
 \end{equation}
 and
\begin{equation}
\label{e:dt2}
  \ \ K_{I_1,I_2} \le 0 \quad \implies \quad  -   \|K_{I_1,I_2} \|_\infty  \le  \textrm{Cov}[ T_{A_1}, T_{A_2} ] \le  \max_{i \in I_1, j \in I_2} K(i,j)   . 
 \end{equation}
\end{proposition}
\begin{remark}
We shall only make use of the upper bound of \eqref{e:dt1} and the lower bound of \eqref{e:dt2}, but we believe the result to be of independent interest.
\end{remark}

Before proving Proposition \ref{p:dt} we recall a classical Gaussian covariance formula (see \cite[Lemma 3.4]{cha08} for the case $f=g$, and the proof in the general case is identical). Let $X'$ denote an independent copy of $X$, and for $t \in [0,\infty)$ define $X^t = e^{-t} X + \sqrt{1 - e^{-2t} } X' $. 
 Then for all absolutely continuous $f(X),g(X) \in L^2$ such that $\|\nabla f(X)\|_2, \|\nabla g(X)\|_2   \in L^2 $,
 \begin{equation}
\label{e:cf}
  \textrm{Cov}[ f(X), g(X) ]  = \int_0^\infty e^{-t} \sum_{1 \le i,j \le n} K(i,j)  \,  \E \Big[  \frac{\partial f(X)}{ \partial X_i}   \frac{\partial g(X^t)}{ \partial X_j}  \Big]   \, dt .  
  \end{equation}

\begin{proof}[Proof of Proposition \ref{p:dt}]
We focus on the upper bound of \eqref{e:dt1}, since the proof of the lower bound and of \eqref{e:dt2} are analogous. By the first and second items of Lemma \ref{l:pt} we may apply \eqref{e:cf} to $f = T_{A_1}$ and $g = T_{A_2}$. This yields
\begin{align*}
\textrm{Cov}[ T_{A_1}, T_{A_2} ]  & =    \int_0^\infty e^{-t}  \sum_{i \in I_1, j \in I_2}K(i,j) \,  \E \Big[  \frac{\partial T_{A_1}(X)}{ \partial X_i}   \frac{\partial T_{A_2}(X^t)}{ \partial X_j}  \Big]   \, dt \\
& \le  \|K_{I_1,I_2} \|_\infty \int_0^\infty e^{-t} \sum_{i \in I_1, j \in I_2} \E \Big[  \frac{\partial T_{A_1}(X)}{ \partial X_i}  \frac{\partial T_{A_2}(X^t)}{ \partial X_j}  \Big]   \, dt \\
& = \|K_{I_1,I_2} \|_\infty \int_0^\infty e^{-t}   \E \Big[ \|   \nabla T_{A_1}(X) \|_1 \| \nabla T_{A_2}(X^t) \|_1 ]    \, dt  \\
& =   \|K_{I_1,I_2} \|_\infty 
\end{align*}
where the inequality used the fact that all terms in the integrand are positive by the assumption $K|_{I_1,I_2} \ge 0$ and the second item of Lemma~\ref{l:pt}, and the final step used the second item of Lemma \ref{l:pt} again.
\end{proof}

\subsubsection{The positively-correlated case}

We first present the proof of Theorem \ref{t:sdi} in the `positively correlated' case that $K_{I_1, I_2} \ge 0$, which is straightforward. In the next subsection we show how to adapt this to the general case.

\smallskip
We will make use of Hoeffding's covariance formula (see \cite[Lemma 2]{leh66})
\begin{equation}
\label{e:hoeff}
 \textrm{Cov}[Y,Z] =  \int_{-\infty}^{\infty}    \int_{-\infty}^{\infty}  \P[ Y \le y, Z \le z] - \P[Y \le y] \P[Z \le z] \, dy dz   
 \end{equation}
  valid for arbitrary $Y, Z \in L^2$. 
  
  \smallskip
  We also make use of the following `local' positive association property: for all $I_1,I_2 \subseteq \{1,\ldots,n\}$ such that $K_{I_1,I_2} \ge 0$, and increasing events $A_1 \in \sigma(I_1)$ and $A_2 \in \sigma(I_2)$, 
 \begin{equation}
 \label{e:pa}
  \P[ A_1 \cap A_2] \ge \P[A_1] \P[A_2] .
 \end{equation}
Eq.\ \eqref{e:pa} is a `local' extension of the standard Gaussian positive association property due to Pitt \cite{pit82}; see \cite[Lemma A.4]{drrv21} for a proof, or one can derive it from \eqref{e:cf} via approximation. Similarly we also have that, for all $I_1,I_2 \subseteq \{1,\ldots,n\}$ such that $K_{I_1,I_2} \le 0$, and increasing events $A_1 \in \sigma(I_1)$ and $A_2 \in \sigma(I_2)$, 
 \begin{equation}
 \label{e:pa2}
  \P[ A_1 \cap A_2] \le \P[A_1] \P[A_2] .
 \end{equation}
One derives \eqref{e:pa2} from \eqref{e:pa} by negating $X|_{I_1}$ and taking the complement of $A_1$.

\begin{proof}[Proof of Theorem \ref{t:sdi} assuming $K_{I_1, I_2} \ge 0$]
Without loss of generality we may assume that $X$ is non-degenerate (otherwise one can consider an approximating sequence $X_n \to X$ in law, since then $\P[X_n \in A] \to \P[X \in A]$ for every increasing event $A$, see \cite[Section 3]{pit82}). We may also assume that $\P[A_i] \in (0,1)$ (otherwise the result is immediate). Note that since  $K_{I_1, I_2} \ge 0$, equation \eqref{e:tsdi2} follows from \eqref{e:pa}, so we focus on~\eqref{e:tsdi1}.

Applying \eqref{e:hoeff} to $Y = T_{A_1}$ and $Z = T_{A_2}$ gives 
\begin{align*}
 \textrm{Cov}[T_{A_1},T_{A_2}]  &=  \int_{-\infty}^{\infty}    \int_{-\infty}^{\infty}  \P[ T_{A_1} \le u, T_{A_2} \le v] - \P[T_{A_1} \le u] \P[T_{A_2} \le v] \, du dv  \\
 & =   \int_{-\infty}^{\infty}    \int_{-\infty}^{\infty}  \P[ X + u \in A_1 , X + v \in A_2 ] - \P[X + u \in A_1] \P[X + v \in A_2] \, du dv \\
 & \ge  \int_0^\eps     \int_0^\eps \P[ X + u \in A_1 , X + v \in A_2 ] - \P[X + u \in A_1] \P[X + v \in A_2] \, du dv
 \end{align*}
where we used the third item of Lemma \ref{l:pt} in the second step, and \eqref{e:pa} in the third step.

Since by Proposition \ref{p:dt} we have $ \textrm{Cov}[T_{A_1},T_{A_2}]  \le \|K_{I_1,I_2}\|_\infty$, we deduce that there exist  $u, v \in [0, \eps]$ such that 
\begin{equation}
\label{e:psdi1}
 \P[ X + u \in A_1 , X + v \in A_2 ] - \P[X + u \in A_1] \P[X + v \in A_2]  \le \| K_{I_1, I_2} \|_\infty / \eps^2 . 
 \end{equation}
Since $A_1$ and $A_2$ are increasing, the left-hand side of \eqref{e:psdi1} is at least 
\[  \P[ X  \in A_1 , X  \in A_2 ] - \P[X + \eps \in A_1] \P[X + \eps \in A_2]  \]
which concludes the proof.
\end{proof}

\subsubsection{The general case}
Note that the `positively-correlated' case only required $K_{I_1,I_2} \ge 0$, and not $K \ge 0$. As such, our strategy to extend to the general case is to approximate $X$ with $X' = X + \sqrt{\kappa} Z$, where $\kappa = \|K_{I_1,I_2}\|_\infty$ and $Z$ denotes an independent standard Gaussian. Since the covariance $K'$ of $X'$ satisfies $K'_{I_1,I_2} \ge 0$, roughly speaking this reduces the proof of \eqref{e:tsdi1} to the previous case. To prove \eqref{e:tsdi2} we use a similar strategy except we replace the constant $Z$ with a vector $\widetilde{Z} = (\widetilde{Z}_i)_i$ satisfying $\widetilde{Z}_i = Z$ if $i \in I_1$ and $\widetilde{Z}_i = -Z$ if $i \in I_2$. Then $X' = X +  \sqrt{\kappa}  \widetilde{Z}$ has covariance $K'$ satisfying $K'_{I_1,I_2} \le 0$.

\begin{proof}[Proof of Theorem \ref{t:sdi}, general case]
As in the previous case we assume that $X$ is non-degenerate and $\P[A_i] \in (0,1)$. As mentioned in Remark \ref{r:inhom} we may also assume that $I_1$ and $I_2$ are disjoint.

We begin with the proof of \eqref{e:tsdi1}. As in the previous case we write $T_{A_i}$ to denote $T_{A_i}(X)$. Suppose there exists $x \in [0,1] $ such that, for all $u,v \in [0, \eps]$, it holds that
\begin{equation}
\label{e:tsdigenp1}
\P[ T_{A_1} \le u, T_{A_2} \le v] - \P[ T_{A_1} \le u] \P[T_{A_2} \le v ] > x   .
\end{equation}
Then for all $u,v \in [\eps/3, 2\eps/3]$,
\begin{align}
\nonumber & \P[ T_{A_1} + \sqrt{\kappa} Z  \le u,  T_{A_1} + \sqrt{\kappa} Z \le v ]  \\
\nonumber & \qquad \ge    \E \Big[ \P \big[ T_{A_1} + \sqrt{\kappa} Z  \le u,  T_{A_1} + \sqrt{\kappa} Z \le v \, \big| \, Z  \big]  \id_{ |Z| \le \eps / (3 \sqrt{\kappa}) } \Big]   \\
\nonumber & \qquad > \E \Big[ \big( x + \P \big[ T_{A_1} + \sqrt{\kappa} Z  \le u \, \big| \, Z \big] \P \big[ T_{A_1} + \sqrt{\kappa} Z \le v \, \big| \, Z  \big] \big)  \id_{ |Z| \le \eps / (3 \sqrt{\kappa}) }   \Big] \\
\label{e:tsdigenp2} & \qquad   \ge x +  \E \Big[  \P \big[ T_{A_1} + \sqrt{\kappa} Z  \le u \, \big| \, Z \big] \P \big[ T_{A_1} + \sqrt{\kappa} Z \le v \, \big| \, Z \big]  \Big] - 2 \P[ |Z| \ge \eps / (3 \sqrt{\kappa}) ] . 
\end{align}
where the second inequality used \eqref{e:tsdigenp1} and the final inequality that $x \in [0,1]$. Now define 
\[ f_1(Z) =  \P \big[ T_{A_1} + \sqrt{\kappa} Z  \le u \, \big| \, Z \big]  \quad \text{and} \quad f_2(Z) =  \P \big[ T_{A_2} + \sqrt{\kappa} Z  \le v \, \big| \, Z  \big ] , \]  
and note that $f_1$ and $f_2$ are decreasing functions of $Z$. Then by positive associations \cite{pit82} (or just the Harris inequality)
\[ \E[ f_1(Z) f_2(Z) ] \ge \E[ f_1(Z) ] \E[f_2(Z)] .\]
Inserting this in \eqref{e:tsdigenp2}, we conclude that, for all $u,v \in [\eps/3, 2\eps/3]$,
\begin{align}
\label{e:psdi2}
\nonumber & \P[ T_{A_1} + \sqrt{\kappa} Z  \le u,  T_{A_1} + \sqrt{\kappa} Z \le v ]    - \P[ T_{A_1} + \sqrt{\kappa} Z  \le u] \P[T_{A_1} + \sqrt{\kappa} Z \le v ]  \\
&  \qquad  > x  - 2 \P[ |Z| \ge \eps / (3 \sqrt{\kappa}) ] .
\end{align}

Now define the vector $X' = X + \sqrt{\kappa} Z$, which has covariance $K' = K + \|K_{I_1,I_2}\|_\infty$, and note that it satisfies $K'|_{I_1, I_2} \ge 0$. Letting $T'_{A_i} = T_{A_i}(X')$, note also that  $T'_{A_i} = T_{A_i} + \sqrt{\kappa} Z$. Hence combining \eqref{e:hoeff}, \eqref{e:pa}, and \eqref{e:psdi2}, we have
\[ \textrm{Cov}[ T'_{A_1} , T'_{A_2} ] \ge  (\eps/3)^2 \big( x - 2 \P[ |Z| \ge \eps / (3 \sqrt{\kappa}) ]  \big) \ge   (\eps/3)^2 \big( x - 18\kappa/\eps^2  \big)  , \]
 where the final step was by Chebyshev's inequality. On the other hand, by Proposition \ref{p:dt},
\[ \textrm{Cov}[ T'_{A_1} , T'_{A_2} ]  \le  2 \kappa . \]
Combining we see that
\[ x < \frac{18\kappa}{\eps^2} +  \frac{18\kappa}{\eps^2} = \frac{36\kappa}{\eps^2} .  \]
We conclude that there exists $u,v \in [0, \eps]$ such that 
\[ \P[ X + u \in A_1 , X + v \in A_2 ] - \P[X + u \in A_1] \P[X + v \in A_2]  \le  \frac{36\kappa}{\eps^2}  , \]
 which, as in the positively-correlated case, yields \eqref{e:tsdi1} (with constant $c = 36$). 
 
We turn to the proof of \eqref{e:tsdi2}, which is similar. Suppose there exists $x \in [0,1] $ such that, for all $u,v \in [-\eps, 0]$, it holds that
\[  \P[ T_{A_1} \le u, T_{A_2} \le v] -  \P[ T_{A_1} \le u] \P[T_{A_2} \le v ]   < -x   . \]
 Then, similarly to in the proof of \eqref{e:tsdi1}, for all $u,v \in [-2\eps/3, -\eps/3]$,
\begin{align}
\nonumber & \P[ T_{A_1} + \sqrt{\kappa} Z  \le u,  T_{A_1} - \sqrt{\kappa} Z \le v ]  \\
\label{e:tsdigenp3} & \qquad   < -x +  \E \Big[  \P \big[ T_{A_1} + \sqrt{\kappa} Z  \le u \, \big| \, Z \big] \P \big[ T_{A_2} - \sqrt{\kappa} Z \le v \, \big| \, Z \big]  \Big]  + 2 \P[ |Z| \ge \eps / (3 \sqrt{\kappa}) ] . 
\end{align}
Redefine
\[ f_1(Z) =  \P \big[ T_{A_1} + \sqrt{\kappa} Z  \le u \, \big| \, Z \big]  \quad \text{and} \quad f_2(Z) =  \P \big[ T_{A_2} - \sqrt{\kappa} Z  \le v \, \big| \, Z  \big ] , \]  
and note that now $f_2$ is a decreasing function of $Z$, so that by positive associations \cite{pit82} $\E[ f_1(Z) f_2(Z) ] \le \E[ f_1(Z) ] \E[f_2(Z)]$. Inserting this in \eqref{e:tsdigenp3}, we conclude that, for all $u,v \in [-2\eps/3, -\eps/3]$,
\begin{align}
\label{e:psdi3}
\nonumber & \P[ T_{A_1} + \sqrt{\kappa} Z  \le u,  T_{A_2} - \sqrt{\kappa} Z \le v ]    - \P[ T_{A_1} + \sqrt{\kappa} Z  \le u] \P[T_{A_2} - \sqrt{\kappa} Z \le v ]  \\
&  \qquad  < -x  + 2 \P[ |Z| \ge \eps / (3 \sqrt{\kappa}) ] .
\end{align}

Now recall that $I_1$ and $I_2$ are assumed disjoint, and introduce a vector $\widetilde{Z}$ satisfying $\widetilde{Z}_i = Z$ if $i = I_1$ and $\widetilde{Z}_i = -Z$ if $i \in I_2$ (with the remaining coordinates arbitrary). Define the vector $X' = X + \sqrt{\kappa}\widetilde{Z} $, which has covariance 
\[  K'(i,j) =  \begin{cases} 
K(i,j)  + \kappa &  \text{ if } i, j \in I_1 \text{ or } i,j \in I_2 , \\
K(i,j)  - \kappa \le 0 &  \text{ if } i \in I_1, j \in I_2 \text{ or } i \in I_2 , j \in I_1 ,
\end{cases} \]
and so in particular $K'_{I_1,I_2} \le 0$. Letting $T'_{A_i} = T_{A_i}(X')$, note also that 
 \[   T'_{A_1} = T_{A_1} + \sqrt{\kappa} Z  \quad \text{and} \quad  T'_{A_2} = T_{A_2} - \sqrt{\kappa} Z  . \]
 Hence combining \eqref{e:hoeff}, \eqref{e:pa2}, and \eqref{e:psdi3}, we have
\[ \textrm{Cov}[ T'_{A_1} , T'_{A_2} ] <  (\eps/3)^2 \big( -x + 2 \P[ |Z| \ge \eps / (3 \sqrt{\kappa}) ]  \big) \le   (\eps/3)^2 \big(- x + 18
\kappa/\eps^2  \big)  . \]
 On the other hand, by Proposition \ref{p:dt} we have
\[ \textrm{Cov}[ T'_{A_1} , T'_{A_2} ]  \ge  - 2 \kappa ,\]
and the conclusion follows as in the proof of \eqref{e:tsdi2} (again with constant $c = 36$).
\end{proof}

  \subsection{Proof of Theorems \ref{t:sdi2} and \ref{t:sdi3}}

The proof of Theorems \ref{t:sdi2} and \ref{t:sdi3} rely on isoperimetric properties of the standard Gaussian space. The basic idea is that `stability' in Gaussian space is dimension-free, being optimised by half-spaces, which essentially reduces the proof of Theorems \ref{t:sdi2} and \ref{t:sdi3} to the two-dimensional case.

\smallskip
Let us state the two properties we need precisely.  Let $Y = (Y_i)_{1 \le i \le n}$ be an i.i.d.\ vector of standard Gaussian random variables, let $Y'$ be an independent copy of $Y$, and let $Z,Z'$ be i.i.d.\ standard Gaussians. Recall that $\Phi(u) = \P[Z \le u]$, and for $\rho \in [-1,1]$ let
\[ \Phi_\rho(u,v) = \P \big[  Z \le  u , \rho Z + \sqrt{1-\rho^2} Z' \le v \big]  \]
be the cdf of the $\rho$-correlated bivariate standard Gaussian vector $(Z, \rho Z + \sqrt{1-\rho^2} Z')$. For $A \subseteq \R^n$ and $\eps \ge 0$, define $A^\eps = \{ x \in \R^n : \text{ there exists $y \in A$ s.t. } \|x-y\|_2 \le \eps \}$. 

\begin{theorem}[Gaussian isoperimetric inequality {\cite[Eq. (3)]{led98}}]
\label{t:gii}
For every Borel set $A \subseteq \R^n$ and $\eps \ge 0$,
\[ \P[ Y \in  A^\eps ] \ge  \Phi \big(  \Phi^{-1}  ( \P[Y \in A]  ) + \eps \big) . \]
\end{theorem}

\begin{theorem}[Gaussian noise stability {\cite{bor85}, \cite[Corollary 4.3]{moo10} }]
\label{t:gns}
For all functions $f_1,f_2 \in L^2$ such that $f_1,f_2 \in [0,1]$, and all $\rho \in [0,1]$,
\[   \E[f_1(Y) f_2(\rho Y + \sqrt{1-\rho^2} Y') ] \le \Phi_\rho\big(    \Phi^{-1}  ( \E[f_1(Y)]  ) , \Phi^{-1}  ( \E[f_2(Y) ] )  \big)   . \]
In particular, for all Borel sets $A_1,A_2 \subset \R^n$ and $\rho \in [0,1]$,
\[ \P[ Y \in  A_1, \rho Y + \sqrt{1-\rho^2} Y' \in A_2 ] \le \Phi_\rho\big(    \Phi^{-1}  ( \P[Y \in A_1]  ) , \Phi^{-1}  ( \P[Y \in A_2]  )  \big) . \]
\end{theorem}

\noindent These theorems adapt to correlated Gaussian vectors as follows:

\begin{corollary}
\label{c:gii}
For all increasing events $A$, and $\eps > 0$,
\[   \P[ X  + \eps \in  A ] \ge   \Phi \big(  \Phi^{-1}  ( \P[X \in A]  ) + \eps / \sqrt{ \|K\|_\infty }\big) .  \]
\end{corollary}
\begin{proof}
Decompose $K = Q^T Q$ for some matrix $Q$, so that $X$ may be represented as $X = Q Y$. Let $t \ge 0$, and let $S \subseteq \R^n$ denote a Borel set such that $\{Y \in S\} = \{X \in A \}$. We claim that
\begin{equation}
\label{e:inc}
 \{Y \in  S^t \} \subseteq \{ X + t \|K\|_\infty \in A \} .
 \end{equation}
Indeed suppose $Y \in S^t$. Then by definition $Y = Y' + v$ for some $Y' \in S$ and $\|v\|_2 \le t$, and so
\[ X +   t \|K\|_\infty = QY + t \|K\|_\infty  = Q Y' + Q v + t \|K\|_\infty   \ge Q Y' \]
 where the inequality is since, by Cauchy-Schwarz, $\|Qv\|_\infty \le t\|Q^T Q\|_\infty  = t\|K \|_\infty $. Since $QY' \in A$ and $A$ is increasing, we conclude that $X +   t \|K\|_\infty \in A$.
 
 Combining Theorem \ref{t:gii} and \eqref{e:inc},
\[  \Phi \big(  \Phi^{-1}  ( \P[X \in A]  ) + t \big)  = \Phi \big(  \Phi^{-1}  ( \P[Y \in S]  ) + t \big)  \le  \P[ Y \in  S^t ]      \le  \P[X + t \|K\|_\infty \in A ]  \]
and the result follows by setting $t = \eps / \|K\|_\infty$.
\end{proof}

\begin{corollary}
\label{c:gns}
For all $I_1,I_2 \subseteq \{1,\ldots,n\}$, and events $A_1 \in \sigma(I_1)$ and $A_2 \in \sigma(I_2)$,
\[   \P[ X \in  A_1 \cap A_2] \le \Phi_{\rho(I_1,I_2) }  \big(    \Phi^{-1}  ( \P[X \in A_1 ]  ) , \Phi^{-1}  ( \P[X \in A_2]  )  \big) .  \]
\end{corollary}
\begin{proof}
This is a slight generalisation of \cite[Corollary 5.2]{mos10}, and we follow its proof. Abbreviate $\rho = \rho(I_1,I_2)$, and $W := X|_{I_1}$, $V  := X|_{I_2}$. Since the conclusion of the corollary is invariant under change of coordinates, without loss of generality (see \cite[Theorem 10.3]{jan97} or the proof of  \cite[Corollary 5.2]{mos10}) we may suppose that $W$ and $V$ are both i.i.d.\ standard Gaussian vectors of equal dimension $k$ such that $\Cov[W,V]=  \rho J$ for $J = (J_{ij})$ a diagonal matrix with entries in $[0,1]$. Then we have  
\[  \P[ (W,V) \in  A_1 \cap A_2]  = \E[   \id_{Y  \in A_1}    g(\rho Y + \sqrt{1-\rho^2} Y' ) ]  \] 
where
\[ g(x_1,\ldots,x_k) =   \P\Big[  \Big( j_{11} x_1 + \sqrt{1 - j_{11}^2 }Y''_1, \ldots  , j_{kk} x_k + \sqrt{1 - j_{kk}^2 }Y''_k  \Big) \in A_2 \Big] , \]
and $Y'' = (Y''_i)_{1 \le i \le k}$ is an independent copy of $Y$. Since $\E[g(Y) ] = \P[Y \in A_2]$, applying Theorem \ref{t:gns} gives the result.
\end{proof}

We next state the analogue of Theorems \ref{t:sdi2} and \ref{t:sdi3} in the two-dimensional case:

\begin{proposition}
\label{p:2dcase}
For all $\rho \in (0,1]$, $u,v \in \R$, and $\eps \ge 0$,
\[  \Phi_\rho(u,v)  \le \Phi(u) \Phi(v + \eps) + e^{- \eps^2 / (8 \rho^2 )} . \]
Moreover, for all $\rho  \in [0, 1)$,
\[ \Phi_\rho(u,v) \le \Phi(u + \kappa \rho) \Phi(v + \kappa \rho)  \ , \quad \kappa =  2 + (1-\rho^2)^{-1/2} \max\{0,  -  \max\{u,v\}   \} .\]
\end{proposition}

The proof of Proposition \ref{p:2dcase} reduces to some standard calculations for bivariate Gaussians. Before giving details, let us finish the proof of Theorems \ref{t:sdi2} and \ref{t:sdi3}:
 
\begin{proof}[Proof of Theorem \ref{t:sdi2}]
First applying Corollary \ref{c:gns}, and then the first statement of Proposition \ref{p:2dcase} with $u = \Phi^{-1}(\P[X \in A_1])$,  $v = \Phi^{-1}(\P[X \in A_2])$ and $\eps \to  \eps / \sqrt{ \|K\|_\infty}$, yields
 \[  \P[ X \in  A_1 \cap A_2]   \le  \P[ X \in A_1] \Phi\big( \Phi^{-1}( \P[X \in A_2 ] ) + \eps / \sqrt{\|K\|_\infty } \big) +  e^{ - \eps^2  / (8 \|K\|_\infty  \rho^2(I_1,I_2) ) }  .\] 
 An application of Corollary \ref{c:gii} completes the proof.
\end{proof}

\begin{proof}[Proof of Theorem \ref{t:sdi3}]
The proof is the same as Theorem \ref{t:sdi2}, except we use the second statement of Proposition \ref{p:2dcase} instead of the first statement, noting that if $\rho \in [0, 1-\delta_1]$ then 
\begin{equation*}
(1-\rho^2)^{-1/2}  \le (1 -  (1-\delta_1)^2)^{-1/2} \ge (\delta_1)^{-1/2} . \qedhere
\end{equation*}
\end{proof}

\noindent It remains to give the proof of Proposition \ref{p:2dcase}:

\begin{proof}[Proof of Proposition \ref{p:2dcase}]
For the first statement, we have
\begin{align*}
 \Phi_\rho(u,v)  & =   \P \big[  Z \le  u , \rho Z + \sqrt{1-\rho^2} Z' \le v ]  \\
 & \le   \P \big[  Z \le  u , \rho Z + \sqrt{1-\rho^2} Z' \le v ,  \rho Z \ge -  \eps/2 \big]  + \P[  \rho Z \le - \eps/2 \big]    \\
 & \le   \P [  Z \le  u  ] \P[ \sqrt{1-\rho^2} Z' \le v + \eps/2] + \P[  \rho Z  \le -  \eps/2 \big]  \\
 & \le \P[ Z \le u] \big(  \P[\rho Z +  \sqrt{1-\rho^2} Z' \le v + \eps]  + \P[  \rho Z \ge  \eps/2 \big]  \big) + \P[  \rho Z \le - \eps/2 \big]  \\
 &  \le \Phi(u) \Phi(v + \eps) +  e^{- \eps^2 / (8 \rho^2 )} ,
 \end{align*}
 where the final step used the equality in law of $Z$ and $\rho Z +  \sqrt{1-\rho^2} Z' $, and the standard Gaussian tail bound in \eqref{e:ssdip1}.
 
We turn to the second statement. Let $\kappa \ge 0$ be as in the statement of the proposition, and for $t \in [0,1]$ define 
\[ (\rho_t, u_t, v_t ) =   \big( t \rho,  u + \kappa \rho(1- t), v + (1-t) \kappa \rho \big) \quad \text{and} \quad \Psi(t) = \Phi_{\rho_t}(u_t,v_t) . \]
Observing that 
\[ \Psi(0) =  \Phi_0(u + \kappa \rho,v + \kappa \rho v) =  \Phi(u + \kappa \rho) \Phi(v + \kappa \rho)    \quad \text{and} \quad  \Psi(1) = \Phi_\rho(u,v) , \]
it remains to show that $\Psi'(t) \le 0$.
 
For $\rho \in (-1,1)$, let $\varphi_{\rho}(u,v)$ denote the pdf of a $\rho$-correlated bivariate standard Gaussian vector, and let $\varphi(u)$ be the standard Gaussian pdf. It is standard that
\[     \frac{\partial \Phi_\rho(u,v) }{\partial u} =  \varphi(u) \Phi( (v-u \rho )/\sqrt{1-\rho^2} )  \quad \text{and} \quad   \frac{\partial \Phi_\rho(u,v) }{\partial v}  =  \varphi(v) \Phi( (u-v \rho )/\sqrt{1-\rho^2} )   .   \]
Indeed this follows from the fact that, by Gaussian regression, $Z_2 | Z_1 = u $ is distributed as $u \rho + \sqrt{1-\rho^2} Z$, and similarly for $Z_1 | Z_2 = v$. It is also standard that
\[   \frac{\partial \Phi_\rho(u,v) }{\partial \rho}  =    \frac{\partial^2 \Phi_\rho(u,v) }{\partial u \partial v}  = \varphi_\rho(u,v) = \varphi(u) \varphi( (v-u \rho )/\sqrt{1-\rho^2} )  =  \varphi(v) \varphi( (u-v \rho )/\sqrt{1-\rho^2} )  . \]
In particular, abbreviating $s_1 = (v-u \rho )/\sqrt{1-\rho^2} $ and $s_2 = (v-u \rho )/\sqrt{1-\rho^2} $,
\begin{equation}
\label{e:derbounds}
 \frac{  \frac{\partial \Phi_\rho(u,v) }{\partial \rho}   } {\max \big\{   \frac{\partial \Phi_\rho(u,v) }{\partial u} ,   \frac{\partial \Phi_\rho(u,v) }{\partial v} \big\}} = \min \Big\{ \frac{\varphi(s^1) }{\Phi(s^1)} ,\frac{   \varphi(s^2) }{ \Phi(s^2) } \Big\} \le 2 + \max\{0, - \max\{s_1,s_2\} \}  ,  
 \end{equation}
where the inequality used a standard bound on the inverse Mill's ratio, valid for all $s \in \R$,
\[ \varphi(s) / \Phi(s) \le \max\{ \varphi(-1) / \Phi(-1) ,  -s -1/s \}  \le 2 + \max\{0,-s\}  . \]

We are now ready to verify that $\Phi'(t) \le 0$. First we note that, by definition, 
\begin{equation}
\label{e:bivar1}
 \rho_t  \le \rho \ , \quad  u_t - \rho_t v_t \ge u - \rho_t v \quad \text{and} \quad v_t - \rho_t u_t \ge v - \rho_t u . 
 \end{equation}
 Also, by breaking into cases depending on the signs of $u$ and $v$, we see that
\begin{equation}
\label{e:bivar2}
 \max_{t \in [0,1]}  \max\Big\{0, - \max\{u - \rho_t v,v - \rho_t u\}  \Big\}  \le \max\Big\{0, - \max\{u,v \}  \Big\}  .   
 \end{equation}
Then by the chain rule,
\[ \Phi'(t)  = \rho \Big(  \frac{\partial \Phi_{\rho_t}(u_t,v_t) }{\partial \rho}   - \kappa  \frac{\partial \Phi_{\rho_t}(u_t,v_t) }{\partial u} - \kappa \frac{\partial \Phi_{\rho_t}(u_t,v_t) }{\partial v} \Big)   , \]
and also, by   \eqref{e:derbounds},
\begin{align*}
   \max_{t \in [0,1]} \frac{  \frac{\partial \Phi_{\rho_t}(u_t,v_t) }{\partial \rho}   } {\max \big\{   \frac{\partial \Phi_{\rho_t}(u_t,v_t) }{\partial u} ,   \frac{\partial \Phi_{\rho_t}(u_t,v_t) }{\partial v} \big\}}  & \le 2 +  \max_{t \in [0,1]} \max\Big\{0,  \frac{- \max\{u_t  - \rho_t v_t ,v_t - \rho_t u_t\} }{ \sqrt{1-\rho_t^2} }  \Big\}   \\
   & \le 2 + (1-\rho^2)^{-1/2}  \max\big\{0, - \max\{u  ,v\}  \big\}  =: \kappa,
   \end{align*}
where the second inequality used \eqref{e:bivar1} and \eqref{e:bivar2}.
\end{proof}

\subsection{Proof of Proposition \ref{p:ssdi}}
Since $A_1$ and $A_2$ are increasing, and using that $X_1|_{I_1}$ and $X_1|_{I_2}$ are independent, 
\begin{align*}
\P[X \in A_1 \cap A_2] & \le   \P[X \in A_1 \cap A_2, X_2|_{I_1 \cup I_2} \le  \eps/2 ]  +  \P\big[  \cup_{i \in I_1 \cup I_2} \{ (X_2)_i \ge  \eps/ 2 \} \big]  \\
&  \le \prod_{i=1,2} \P[X_1  + \eps/2 \in A_i]  +  \P\big[  \cup_{i \in I_1 \cup I_2} \{ (X_2)_i \ge  \eps/ 2 \} \big]   \\
&  \le \prod_{i=1,2} \P[X_1  + \eps/2 \in A_i]  +  2 \max\{  |I_1|, |I_2| \}  \max_{i \in I_1 \cup I_2}  \P[ (X_2)_i \ge \eps/ 2 ]  .
\end{align*}
Similarly,
\begin{align*} 
  \prod_{i=1,2} \P[X_1  + \eps/2 \in A_i]  & \le   \prod_{i=1,2} \Big( \P[X  + \eps \in A_i]   +  \P\big[  \cup_{i \in I_2} \{ (X_2)_i \le  -\eps/ 2 \} \big]  \Big) \\
   & \le   \prod_{i=1,2} \P[X + \eps \in A_i] + 3 \max\{  |I_1|, |I_2| \}  \max_{i \in I_1 \cup I_2}  \P[ (X_2)_i \le -\eps/ 2 ]  .
\end{align*}
Combining, and using that $X_2$ is centred, 
\[  \P[X \in  A_1 \cap A_2] -   \P[X + \eps \in A_1]\P[X + \eps \in A_2] \le 5 \max\{  |I_1|, |I_2| \} \max_{i \in I_1 \cup I_2}  \P[ (X_2)_i \ge \eps/ 2 ] .\]
To conclude recall that $\text{Var}[X_2(i)] \le \sigma^2 $ for all $i \in I_1 \cup I_2$. Letting $Z$ denote a standard Gaussian, we have for $i \in I_1 \cup I_2$ and all $t \ge 0$,
\begin{equation}
\label{e:ssdip1}
   \P[ (X_2)_i \ge t ] \le \P[ Z \ge t/ \sigma ]  \le \frac{ e^{- t^2 /  (2 \sigma^2)  }}{2}  ,
  \end{equation}
where the last step is a standard Gaussian tail bound.  Setting $t = \eps/2$ gives the result.

\subsection{Proof of Proposition \ref{p:neg}}
Suppose \eqref{e:sdineg} were true, let $X = (Z, Z)$ where $Z$ is a standard Gaussian random variable, and fix $\kappa \in (0,1)$ to be chosen later. Then \eqref{e:sdineg} applied to the events $A_1 = A_2 = \{ Z \ge u\}$ and sprinkling parameter $\eps = \kappa u$ implies that, for all $u > 0$,
\begin{equation}
\label{e:sdineg2}
   \P[Z \ge u] \le  \P[ Z \ge u (1-\kappa) ]^2 + c_1 e^{- c_2 \kappa^2 u^2 } . 
   \end{equation}
Using the standard fact that, as $u \to \infty$,
\[ \P[Z \ge u] \sim  \frac{1}{\sqrt{2 \pi} u } e^{-u^2 /2 }  , \]
  taking $u \to \infty$ in \eqref{e:sdineg2} shows that $1/2 \ge \min \{ 1 - 2 \kappa + \kappa^2   , c_2 \kappa^2 \} $.
  In particular, if $\kappa > 1 - \sqrt{1/2}$ then $c_2 \le 1/ (2 \kappa^2)$. Hence taking $\kappa \downarrow 1 - \sqrt{1/2} \approx 0.293\ldots$ \ yields that $c_2 \le 3 + 2 \sqrt{2} \approx 5.828\ldots$

\bigskip
\section{Application to level-set percolation of Gaussian fields}
\label{s:per}

In this section we give an application of Theorem \ref{t:sdi} in Gaussian percolation theory. Let $f$ be a Gaussian field on either $T=\Z^d$ or $T=\R^d$, $d \ge 2$, and if $T = \R^d$ assume that $f$ is continuous. For $\ell \in \R$, let $\P_\ell$ denote the law of $f + \ell$. Gaussian percolation theory is the study of the phase transition in the global connectivity of the excursion sets $\{f + \ell  \ge 0\}$ as $\ell$ increases. It is natural to define a critical parameter $\ell_c \in [-\infty,\infty]$ as
\[ \ell_c = \inf \Big\{ \ell \in \R : \P_\ell \big[ \{f \ge 0\} \text{ contains an unbounded path-connected component} \big] > 0 \Big\}  ,\]
where \textit{path} refers to a lattice path if $T = \Z^d$ and a continuous path if $T = \R^d$. A central question in the theory is whether the phase transition is \textit{non-trivial}, i.e.\ whether $\ell_c \in (-\infty,\infty)$, and one expects this to be true in wide generality. 

\subsection{Conditions for non-triviality}

Our main result gives sufficient conditions for non-triviality. For $r > 0$, let $B(r)$ denote the Euclidean ball of radius $r$ centred at the origin.

\begin{theorem}
\label{t:nontriv}
If both the following conditions are satisfied then $\ell_c \in (-\infty,\infty)$:
\begin{enumerate}
\item (Uniformly bounded local suprema) 
\[  \sup_{x \in T} \E \Big[ \sup_{  y \in x + B(1) \cap T }  |f(y)| \Big] < \infty .\]
\item (Polylogarithmic correlation decay) There exist constants $c, \delta > 0$ such that 
\[ |\Cov[f(x),f(y)] | \le  g(\|x-y\|_2) , \quad x,y \in T ,\]
where $g(r) =  c  (\log (1+r))^{-3-\delta}$.
\end{enumerate}
\end{theorem}

\noindent Theorem \ref{t:nontriv} establishes non-triviality in very wide generality:

\begin{example}
The first condition of Theorem \ref{t:nontriv} is satisfied if either:
\begin{enumerate}
\item $f$ is stationary;
\item $T = \Z^d$ and both $\E[f(x)]$ and $\Var[f(x)]$ are bounded; or
\item $T = \R^d$, $\E[f(x)]$ and $\Var[f(x)]$ are bounded, and $K(x,y)$ is H\"{o}lder continuous.
\end{enumerate}
\end{example}

\begin{example}[Monochromatic random waves]
\label{e:mrw}
Suppose $f$ is the \textit{monochromatic random wave} on $\R^d$, $d \ge 2$, that is, the centred isotropic Gaussian field with covariance 
\[ K(0,x) = \widehat{\mu_{\mathbb{S}^{d-1}}}(x) =  \|x\|_2^{-d/2 +1} J_{d/2 - 1}(\|x\|_2) = O\big( \|x\|_2^{-(d-1)/2}\big), \]
where $\widehat{\mu_{\mathbb{S}^{d-1}}}$ denotes the Fourier transform of the normalised Lebesgue measure on the sphere, and $J_n$ is the order-$n$ Bessel function. Then the conditions of Theorem \ref{t:nontriv} are satisfied and $\ell_c \in (-\infty,\infty)$. Previously this was only known in the case $d=2$ \cite{mrv20} (c.f.\ Remark \ref{r:riv}).
\end{example}

The question of non-triviality for Gaussian percolation models has received substantial attention in the literature, especially in the stationary setting. Early works on this topic were \cite{ms83a,ms83b,ms83c}, which proved non-triviality for stationary fields with bounded spectral density, including fields with $K(0,\cdot) \in L^1$. For strongly-correlated fields, non-triviality was first established for the GFF on $\Z^d$, $d \ge 3$, \cite{rs13}, using a sprinkled decoupling inequality similar to in Proposition~\ref{p:ssdi}. Recently this has been extended to a wider class of stationary strongly-correlated fields which satisfy a decomposition of the form \eqref{e:decomp} \cite{cn21,ms22}, including the Cauchy fields in Example \ref{e:cauchy}, as well as to many isotropic \textit{planar} fields using techniques specific to the planar case \cite{mrv20}.

\begin{remark}[Optimality of the decay assumption]
We do not expect that the polylogarithmic decay exponent $\gamma > 3$ in Theorem \ref{t:nontriv} is optimal. Indeed in \cite{ms22} non-triviality was established for a class of stationary fields with polylogarithmic decay with exponent $\gamma > 1$, and in \cite{mrv20} for smooth isotropic fields on $\R^2$ whose correlation decay is of order $(\log \log r)^{-2-\delta}$. In Theorem~\ref{t:subcrit} below we show that polylogarithmic decay with exponent $\gamma > 2$ is sufficient to conclude that $\ell_c > -\infty$. See Section \ref{s:proofs} for an informal discussion on the role of the decay exponents $\gamma > 2$ and $\gamma > 3$ to conclude $\ell_c > - \infty$ and $\ell_c < \infty$ respectively.

\smallskip
\noindent It is natural to ask whether \textit{qualitative} mixing conditions are sufficient for non-triviality:
\end{remark}

\begin{question}
Suppose $f$ satisfies the uniform boundedness condition in Theorem \ref{t:nontriv}, and $\Cov[f(x),f(y)] \le g(\|x-y\|_2)$ for some $g(r)  \to 0$ as $r \to \infty$. Then is $\ell_c \in (-\infty,\infty)$? What if we instead assume that $f$ is stationary and ergodic?
\end{question}

\begin{remark}
\label{r:riv}
Alejandro Rivera communicated to us an alternative proof of non-triviality for the monochromatic random waves in Example \ref{e:mrw}, based on the observation that non-sprinkled decoupling bounds of the form
\[       \big| \P[  A_1 \cap A_2    ] - \P[A_1] \P[A_2]  \big| \le c \,  \text{dist}(D_1,D_2)^{2d} \| K_{D_1,D_2} \|_\infty  ,\]
such as those appearing in \cite{bmr20}, are sufficient to prove non-triviality if the covariance $K$ decays polynomially in the distance; this is similar to the argument developed for the Poisson cylinder model in \cite{tw12}. However, as well as requiring stronger decay than in Theorem \ref{t:nontriv}, this argument gives a weaker quantitative conclusion than we obtain (in Section \ref{s:subcrit} below).
\end{remark}

\subsection{Rate of subcritical decay of connectivity}
\label{s:subcrit}
The proof of Theorem \ref{t:nontriv} also provides quantitative bounds on the rate of connectivity decay in the subcritical regime. 

\smallskip
For $A,B \subset \R^d$, let $\{A \longleftrightarrow B\}$ denote the event that there is a path (a lattice path if $T = \Z^d$ and a continuous path if $T = \R^d)$ in $\{f \ge 0\}$ that intersects $A$ and $B$. Define 
\begin{equation}
\label{e:tildelc}
\widetilde{\ell}_c = \sup \Big\{ \ell \in \R :   \liminf_{R \to \infty} \sup_{x \in \R^d} \P_\ell[ x + B(R) \longleftrightarrow x + \partial B(2R) ] = 0  \Big\} .
\end{equation}
By countable additivity it is clear that $\widetilde{\ell}_c \le \ell_c$, and it is expected that $\widetilde{\ell}_c = \ell_c$ in wide generality, although this has only been verified in certain special cases \cite{dgrs20, mrv20, m22}.

\begin{theorem}
\label{t:subcrit}
Let $f$ satisfy the first condition of Theorem \ref{t:nontriv}, and let $g: \R^+ \to \R^+$ be a decreasing function satisfying 
\[ |\Cov[f(x),f(y)] | \le   g(\|x-y\|_2) , \quad x,y \in T . \]
Assume there exists a $\delta > 0$ and a decreasing function $h': \R^+ \to \R^+$ satisfying, as $r \to \infty$,
\[   h(r) :=   g(r)(\log r)^{2 +\delta} \to 0 , \quad    h(r) / h'(25 r) < \infty,   \quad \text{and} \quad  h'(r)^2 / h'(5r) \to 0 . \]
Then $\widetilde{\ell}_c > -\infty$, and for every $\ell < \widetilde{\ell}_c$ and there exists $c > 0$ such that
 \begin{equation}
 \label{e:tsubcrit1}
  \sup_{x \in \R^d} \P_\ell[  x   \longleftrightarrow  x + \partial  B(R) ]  \le c h'(R)  , \quad R \ge 2. 
  \end{equation}
\end{theorem}

\begin{example}
To clarify the conditions in Theorem \ref{t:subcrit}, let us give some examples:
\begin{enumerate}
\item If $g(r) = (\log r)^{-\gamma}$, $\gamma > 2$, or $g(r) = r^{-\alpha}$, $\alpha > 0$, one may take $h' = h$. In particular this shows that polylogarithmic decay with exponent $\gamma > 2$ is sufficient for $\ell_c > -\infty$.
\item If $h(r) = g(r)(\log r)^{2 +\delta} = e^{-c r^\beta}$, $\beta \in (0,1]$, one may take $h'(r) = e^{-(c/25) r^\beta }$.
\end{enumerate}
Note that one can never take $h'$ decaying faster than exponential. This is natural, since if $f$ is stationary and $K \ge 0$, then $\P_\ell[  0   \longleftrightarrow  \partial  B(R) ]  \ge e^{-cR}$ by positive associations.
\end{example}

\begin{example}[Random plane wave]
Let $f$ be the monochromatic random wave from Example~\ref{e:mrw} in dimension $d=2$. Then it is known \cite{mrv20} that $\ell_c = \widetilde{\ell}_c = 0$, and Theorem \ref{t:subcrit} implies that, for every $\ell < 0$ and $\delta > 0$ there exists a $c > 0$ such that
\[  \P_\ell[  0  \longleftrightarrow \partial  B(R) ]   \le  c R^{-1/2} (\log R)^{2+\delta}  , \quad R \ge 2. \]
This is the first polynomial bound on the subcritical connectivity decay of the random plane wave; previously only the weaker bound $c_1 e^{- c_2 \sqrt{\log R} } $ was known \cite{mrv20}.
\end{example}

\begin{remark}
In the examples in \cite{pr15,grs21,ms22} the stronger bound 
\begin{equation}
\label{e:subcritstrong}
  \P_\ell[  0  \longleftrightarrow \partial  B(R) ]   \le  c_1 e^{ - c_2 \min\{R, 1/ g(R) \}  } 
  \end{equation}
was established for all $\ell < \widetilde{\ell}_c$ by exploiting a decoupling inequality similar to \eqref{e:ssdi2} (more precisely this gives \eqref{e:subcritstrong} up to a logarithmic factor in the exponent, with further analysis needed to remove this factor). 
\end{remark}

\begin{question}
Does the bound \eqref{e:subcritstrong} hold under the conditions of Theorem \ref{t:subcrit}? 
\end{question}

\begin{remark}
A postiori, under the assumptions of Theorem \ref{t:subcrit} one can replace the $\liminf$ in \eqref{e:tildelc} with a limit without change to the value of $\widetilde{\ell}_c$, although this is not clear in general.
\end{remark}

\subsection{Proof of Theorems \ref{t:nontriv} and \ref{t:subcrit}}
\label{s:proofs}
The proofs rely on a variant of the Kesten bootstrap \cite[Section 5]{kes82}; although similar arguments have appeared elsewhere (e.g.\ \cite{pt15,mv20,ms22}), let us begin by giving a brief outline of the method.

For every $x \in \R^d$ and $R > 0$, let $E_{x,R}$ denote the `annulus crossing' event $\{x + B(R) \longleftrightarrow x + \partial B(2R)\}$. We aim to find a bound on $ \sup_{x \in \R^d}  \P[E_{x,5R}]$ in terms of the \textbf{square} of $ \sup_{x \in \R^d}  \P[E_{x,R}]$, which by iterating along a geometric sequence of scales will yield a rapidly decaying bound on $\sup_{x \in \R^d} \P[E_{x,R}]$ (provided the initial scale is chosen correctly).

The key observation is the following: for every $x \in \R^d$ and $R > 0$ one may choose two collections of points $(x_i)_{1 \le i \le n_d}$ and $(y_j)_{1 \le j \le n_d}$, where $n_d > 0$ depends only on the dimension, such that $\|x_i - y_j \|_2 \ge R$ for all $i,j$, and
\[  E_{x,5R} \implies \cup_{i,j} \{ E_{x_i, R}  \cap E_{y_j, R} \} . \]
  Hence, by the union bound and the sprinkled decoupling inequality \eqref{e:tsdi1} (or \eqref{e:tcsdi1} in the case $T = \R^d$), we establish the sprinkled bootstrapping inequality
  \begin{equation}
\label{e:boot}
  \sup_{x \in \R^d}  \P_{\ell-\eps}[   E_{x,5R} ]  \le  n_d^2  \Big( \sup_{x \in \R^d} \P_{\ell}[ E_{x,R}]^2 +  c g(R) \eps^{-2} \Big)  , \quad  R > 0, \ell \in \R, \eps > 0 .
  \end{equation}
We obtain Theorems \ref{t:nontriv} and \ref{t:subcrit} from a deterministic analysis of \eqref{e:boot}, similar to in \cite{pt15,mv20,ms22}. Before embarking on this, let us give an informal explanation of the role of polylogarithmic decay exponents $\gamma > 3$ and $\gamma > 2$ in analysing \eqref{e:boot}.

Observe two key features of \eqref{e:boot}: (i) one must `sprinkle' the level from $\ell$ to $\ell-\eps$ when moving up a scale;  and (ii) there is an additive error $c g(R) \eps^{-2}$ (the multiplicative error $n_d^2$ plays no role). To ensure the sprinkling does not send the level to $-\infty$, we need it to be summable over the scales, i.e.\ we need $\eps = \eps_R \approx (\log R)^{-1-\delta/2}$. This choice of $\eps$ makes the additive error $\approx c g(R)(\log R)^{2+\delta} =: h(R)$. Since this must tend to zero, we require $g(R) \ll (\log R)^{-2-\delta}$ as in Theorem \ref{t:subcrit}. Given this, the output of the bootstrap is a bound on crossing probabilities of the same order as the additive error $h(R) = g(R)(\log R)^{2+\delta}$, which is roughly the content of Theorem \ref{t:subcrit} (and also yields $\ell_c > -\infty$). To further establish $\ell_c < \infty$, we need in addition that crossings probabilities are summable over the scales to allow for a Borel-Cantelli argument. This requires $h(R) \ll (\log R)^{-1-\delta}$, and hence $g(R) \ll (\log R)^{-3-2\delta}$ as in Theorem \ref{t:nontriv}.

\begin{proof}[Proof of Theorem \ref{t:subcrit}]
    For later use we observe that, by symmetry, the union bound, and Markov's inequality, for all $\ell < 0$ and $R \ge 1$, 
\begin{equation}
\label{e:ebound}
 \sup_{x \in \R^d} \P_\ell[E_{x,R}] \le   \sup_{x \in \R^d} \P\Big[ \sup_{y \in x + B(2R)} f(y) \ge - \ell \Big] \le  \frac{c_d R^d \sup_{x \in \R^d} \E \Big[ \sup_{y \in x + B(1)} |f(y)| \Big] }{|\ell|}    .
  \end{equation}
  In particular, by the assumption of uniformly bounded local suprema, if $R$ is fixed then we can ensure that $ \sup_{x \in \R^d} \P_\ell[E_{x,R}]$ is arbitrary small by taking $\ell$ sufficiently small.
  
We first show $\widetilde{\ell}_c > -\infty$. Define $\ell' \in \R$ and $R_0 > 1$ to be to be determined later, and the decreasing sequence $(\ell_n)_{n \ge 1}$ satisfying
 \begin{equation}
 \label{e:ell}
 \ell_1 = \ell'  \quad \text{and} \quad \ell_{n+1} = \ell_n - (\log (R_0 5^n ))^{-1-\delta/2}   .
 \end{equation}
 Observe that $\ell_\infty = \lim_{n \to \infty} \ell_n  > -\infty$. For $n \in \N$, define $p_n = \sup_{x \in \R^d}  \P_{\ell_n}[ E_{x, R_0 5^n} ] $, which by \eqref{e:boot} satisfies
  \begin{equation}
  \label{e:induct1}
     p_{n+1} \le n_d^2 \big( p_n^2 +    g( R_0 5^n) (\log (R_0 5^n ))^{2+\delta}  \big) =  n_d^2 p_n^2 +    c n_d^2 h(R_0 5^n) . 
     \end{equation}
  Now recall we assume that 
  \begin{equation}
  \label{e:induct2}
  h(r) \le c' h'(25r) 
  \end{equation}
   for some $c' > 0$, and also that $ h'(r)^2 / h'(5r) \to 0$ as $r \to \infty$. This allows us to fix $R_0$ sufficient large so that
  \begin{equation}
  \label{e:r0}
    h'(r)^2 / h'(5r) \le (4 n_d^4 c c')^{-1}   \ , \quad \text{for all } r \ge R_0 .
    \end{equation}
    Moreover, by the discussion following \eqref{e:ebound} we may fix $\ell' \in \R$ sufficiently small so that 
  \begin{equation}
  \label{e:base} 
  p_1 \le  2 n_d^2 c c'  \times h'(25 R_0) .  
  \end{equation}
Using \eqref{e:induct1}--\eqref{e:base}, an inductive argument then shows that
  \[ p_n \le  2 n_d^2 c c'  \times h'(5 \times R_0 5^n) \]
  for all $n \in \N$. Since $h'(r) \to 0$ we deduce that $p_n \to 0$, which by monotonicity implies that $\widetilde{\ell}_c \ge \ell_\infty  > -\infty$.

 Let us now prove \eqref{e:tsubcrit1}. Let $\ell < \widetilde{\ell_c}$ be given, and fix $\ell' \in (\ell,  \widetilde{\ell_c})$ arbitrarily.  Let $R_0 > 1$ be sufficiently large so that \eqref{e:r0} holds, and also so that the decreasing sequence $(\ell_n)_{n \ge 1}$ defined via \eqref{e:ell} satisfies $\ell_\infty = \lim_{n \to \infty} \ell_n > \ell$. Define $p_n = \sup_{x \in \R^d}  \P_{\ell_n}[ E_{x, R_1 5^n} ] $ for a $R_1 \ge R_0$ to be determined later. Then by \eqref{e:boot} we have
 \[     p_{n+1} \le  n_d^2 p_n^2 +   c n_d^2 g(R_1 5^n) (\log (R_0 5^n ))^{2+\delta} \le  n_d^2 p_n^2 +   c n_d^2 h(R_0 5^n) , \]
 where the second inequality used that $g$ is decreasing. Since $\ell' < \widetilde{\ell}_c$, we may choose $R_1$ sufficiently large so that \eqref{e:base} holds. Hence by induction we have again that
  \[ p_n \le  2 n_d^2 c c'  \times h'(5 \times R_0 5^n)  .\]
 Since $h'$ is decreasing, by monotonicity this gives the result.
 \end{proof}
 
 \begin{proof}[Proof of Theorem \ref{t:nontriv}]
 Without loss of generality we may assume that $d=2$. In this context it is more convenient to work with crossings of rectangles rather than annuli, so we define
  \[ \hat{\ell}_c =  \inf \Big\{ \ell \in \R :   \limsup_{R \to \infty} \inf_{x \in \R^d} \min\big\{  \P_\ell[ x + \textrm{HCross}(R)  ] , \P_\ell[ x + \textrm{VCross}(R)  ]   \big\}  = 1  \Big\} , \]
 where $ \textrm{HCross}(R)$ is the `horizontal box-crossing' event that there is a path in $\{f \ge 0\}|_{[0,5R] \times [0,R]}$ that intersects $\{0\} \times [0,R]$ and $\{5R\} \times [0, R]$, and $\textrm{VCross}(R)$ is the  `vertical' analogue with the coordinates interchanged. Using a similar argument to in the proof of \eqref{e:tsubcrit1} one can show that $\hat{\ell}_c < \infty$, and further that for all $\ell > \hat{\ell}_c$ we have
\[  \inf_{x \in \R^d} \min \big\{  \P_\ell[ x + \textrm{HCross}(R)  ] , \P_\ell[ x + \textrm{VCross}(R)  ]   \big\}  \ge 1 - c (\log R)^{-1-\delta'}  \]
for some $c,\delta' > 0$. Observe finally that, for every $n_0 \ge 1$,
\[  \bigcap_{n \ge n_0}   \big(  \textrm{HCross}(5^n)   \cap   \textrm{VCross}(5^n)    \big)   \subseteq \{f \ge 0\} \text{ has an infinite component} . \]
By the Borel-Cantelli lemma we deduce that $\ell_c \le \hat{\ell}_c < \infty$, which completes the proof.
\end{proof}

\bigskip
\section{Consequences for non-sprinkled decoupling}
\label{s:nsd}

In this section we discuss consequences of Theorem \ref{t:sdi} for non-sprinkled decoupling inequalities of the form \eqref{e:nsdi}. These arise by combining Theorem \ref{t:sdi} with some a priori control on the stability of $\P_{\ell}[A]$ under perturbations of $\ell$. Here we consider two stability estimates -- (i) in terms of the capacity, and (ii) for the class of `topological events' -- and we believe that other types of stability estimates may also give interesting consequences.

\subsection{Stability via the capacity}
Recall that $X = (X_i)_{1 \le i \le n}$ is a Gaussian vector with covariance $K$. For $I \subseteq \{1,\ldots,n\}$, define the \textit{capacity} of~$I$
\begin{equation}
\label{e:cap}
\Capa(I) = \Capa_K(I) =   \Big( \inf_{\mu \in \mathcal{P}(I)} \sum_{i,j \in I} \mu(i) K(i,j) \mu(j) \Big)^{-1}  \in [0, \infty ] , 
\end{equation}
where $\mathcal{P}(I)$ is the set of probability measures on $I$. As a consequence of the Cameron-Martin theorem (see Section \ref{s:stabproof} for details), one has the following stability estimate for increasing events: 

\begin{proposition}
\label{p:cap}
For every $I \subseteq \{1,\ldots,n\}$, increasing event $A \in \sigma(I)$, and $\eps > 0$,
\[    \P[X + \eps \in A ] - \P[X \in A]  \le \frac{\eps \sqrt{\Capa(I) }}{2} .    \]
\end{proposition}

\noindent Combining with Theorem \ref{t:sdi} yields:

\begin{corollary}
\label{c:cap}
 There exists a universal constant $c > 0$ such that, for all $I_1,I_2 \subseteq \{1,\ldots,n\}$, and increasing events $A_1 \in \sigma(I_1)$ and $A_2 \in \sigma(I_2)$,
\[ \Big| \P[A_1 \cap A_2] - \P[A_1]\P[A_2]  \Big| \le  c \Big( \sqrt{   \Capa(I_1) \Capa(I_2) }  \|K_{I_1,I_2}\|_\infty \Big)^{1/3}   . \]
\end{corollary}
\begin{proof}
Note the trivial inequality $cd - ab \le 2( (c-a) + (d-b))$ for all $a,b,c,d \in [0,1]$ with $c \ge a$ and $d \ge b$. Combining with Proposition \ref{p:cap}, for all $\eps_1,\eps_2 > 0$,
\begin{align*}
&  \P[X \in A_1 \cap A_2] - \P[X \in A_1]\P[X \in A_2]  \\
 & \qquad \le  \P[X \in A_1 \cap A_2] - \P[X + \eps_1 \in A_1]\P[X + \eps_2 \in A_2]  +  \eps_1 \sqrt{\Capa(I_1) } +  \eps_2 \sqrt{\Capa(I_2) }  .
 \end{align*}
 Using the inhomogeneous form of equation \eqref{e:tsdi1} (see Remark \ref{r:inhom}), the above is at most
\[ c  \|K_{I_1,I_2}\|_{\infty} / (\eps_1 \eps_2) +  \eps_1 \sqrt{\Capa(I_1) }  +  \eps_2 \sqrt{\Capa(I_2) } . \]
Setting 
 \[ \eps_1 =  \frac{(c \sqrt{\Capa(I_2)}  \|K_{I_1,I_2}\|_\infty )^{1/3} }{\Capa(I_1)^{1/3} }  \quad \text{and} \quad \eps_2 =  \frac{(c \sqrt{\Capa(I_1)}  \|K_{I_1,I_2}\|_\infty )^{1/3} }{\Capa(I_2)^{1/3} } \]
  yields
 \[ \P[A_1 \cap A_2] - \P[A_1]\P[A_2]  \le  \Big( c \sqrt{   \Capa(I_1) \Capa(I_2) }  \|K_{I_1,I_2}\|_\infty \Big)^{1/3} . \]
The reverse inequality is proven similarly, using \eqref{e:tsdi2} in place of \eqref{e:tsdi1}.
\end{proof}

\begin{remark}
Proposition \ref{p:cap} and Corollary \ref{c:cap} also hold for continuous processes, with the obvious changes to notation. Note however that in the continuous setting $\Capa(D)$ may be infinite for compact $D \subset \R^d$ even if the process is non-degenerate (see Example~\ref{e:mrw2} below).
\end{remark}

\begin{example}[Gaussian free field]
Suppose $X$ is the GFF on $\Z^d$, $d \ge 3$, as in Example \ref{e:gff}. Then $\Capa(I)$ coincides with the usual harmonic capacity. In particular, fixing $k > 2$, if $I_1$ and $I_2$ are translations of the Euclidean ball $B(R)$ of radius $R \ge 1$ restricted to $\Z^d$, with centres at least $k R$ apart, then $\Capa(I_i) \sim c_d' R^{d-2}$, and by Corollary \ref{c:cap} we have
\[ \sup_{ \stackrel{A_1 \in \sigma(RI_1), A_2 \in \sigma(R I_2)}{A_i \text{ increasing} } } \Big| \P_{\ell}[A_1 \cap A_2] - \P_{\ell}[A_1]\P_{\ell}[A_2] \Big| \le  c_{d,k} \]
for some $c_{d,k} \to 0$ as $k \to \infty$.
\end{example}

\begin{example}[Short-range fields]
Suppose $f$ is a stationary continuous Gaussian field on $\R^d$ such that $K(0,\cdot)$ is absolutely integrable and $  \widehat{K}(0) :=  \int_{x \in \R^d} K(0,x) dx  > 0 $. Then the capacity has volume scaling (see \cite[Proposition 2.4]{ms22} for the $d=1$ case, and the general case is similar), i.e.\ for every smooth compact domain $D \subset \R^d$,
\[   \Capa(RD) \sim  \frac{\textrm{Vol}(D) R^d}{ \widehat{K}(0) }     , \quad R \to \infty .  \]
 Hence if also $\|K(0,x)  \|x\|_2^d \to 0$ as $\|x\|_2 \to \infty$, and if $D_1 , D_2 \subset \R^d$ are disjoint smooth compact domains, by Corollary \ref{c:cap} we have
 \[   \lim_{R \to \infty}  \sup_{ \stackrel{A_1 \in \sigma(RD_1), A_2 \in \sigma(R D_2)}{A_i \text{ increasing} } }  \Big| \P_{\ell}[A_1 \cap A_2] - \P_{\ell}[A_1]\P_{\ell}[A_2]  \Big| = 0  .\]
This property is sometimes known as `quasi-independence', see e.g.\ \cite{bg17,rv19,mv20}.
\end{example}

\begin{example}[Monochromatic random waves]
\label{e:mrw2}
Suppose $f$ is the monochromatic random wave from Example \ref{e:mrw}. Then there exists $r_0 = r_0(d) > 0$ such that, for every $D \subset \R^d$ which contains a translation of the ball $B(r_0)$, $\Capa(D) = \infty$.
\end{example}

To the best of our knowledge Corollary \ref{c:cap} is new, but in the case of the GFF a stronger version is known. Recall the maximum correlation coefficient $\rho(I_1,I_2)$, which satisfies
\begin{equation}
\label{e:alpharho}
  \sup_{ A_1 \in \sigma(I_1), A_2 \in \sigma(I_2) }    \big| \P[A_1 \cap A_2] - \P[A_1]\P[A_2]  \big| \le \rho(I_1,I_2)  .    
  \end{equation} 

\begin{proposition}[{\cite[Proposition 1.1]{pr15}}]
\label{p:cor}
Suppose $f$ is the GFF on $\Z^d$, $d \ge 3$. Then for all $I_1,I_2 \subseteq \{1,\ldots,n\}$, 
\begin{equation}
\label{e:cap2}
 \rho(I_1,I_2) \le    \sqrt{   \Capa(I_1) \Capa(I_2) }  \|K_{I_1,I_2}\|_\infty   . 
 \end{equation}
  \end{proposition}

\begin{remark}
 Although Proposition \ref{p:cor} and \eqref{e:alpharho} imply a stronger decoupling bound than Corollary \ref{c:cap} (and applying to all events, not only increasing), this bound is non-trivial in the same regime $\sqrt{   \Capa(I_1) \Capa(I_2) } \ll  \|K_{I_1,I_2}\|_\infty$ as the fully general Corollary \ref{c:cap}.
\end{remark}

\begin{remark}
By setting $\alpha$ and $\beta$ in \eqref{e:rholin} to be respectively the measures that achieve the infimum in the definition \eqref{e:cap} of $\Capa(I_1)$ and $\Capa(I_2)$, it is easy to see that, in general,
\begin{equation}
\label{e:caplower}
\rho(I_1,I_2) \ge   \sqrt{   \Capa(I_1) \Capa(I_2) } \min_{i \in I_1, j \in I_2} K(i,j) .   
\end{equation}
\end{remark}

\begin{question}
Does \eqref{e:cap2} hold in general? What about if $K \ge 0$?
\end{question}

\subsection{Stability for topological events}
We next restrict to the case of smooth Gaussian fields $f$ on $\R^d$ and events $A$ which depend only on the \textit{topology} of the excursion sets $\{f \ge u\}$; following \cite{bmr20} we call these `topological events'. For such events, the stability of $\P_{\ell}[A]$ is induced by the absence of critical points which have critical level $\approx u$.

\smallskip
To make this precise, let us introduce some notation. A \textit{box} $B \subset \R^d$ is a compact domain of the form $[a_i ,b_i] \times \ldots \times [a_d, b_d]$ for finite $a_i < b_i$. We consider a box to be equipped with its canonical \textit{stratification}, i.e.\ the partition of $B$ into the collection $\mathcal{S} = (S_i)_i$ of the interiors of each of its faces of dimension $n \in \{0, \ldots, d\} $, which we refer to as \textit{strata}. Each strata of dimension $> 0$ is equipped its Lebesgue measure, and each zeroth-dimensional strata equipped with the counting measure. Define $\overline{\textrm{Vol}}(B) = \sum_{S_i \in \mathcal{S}} \textrm{Vol}(S_i)$.

\smallskip
We assume that $f$ is $C^2$-smooth and that $(f(x), \nabla f(x))$ is non-degenerate for every $x \in \R^d$. Then by Bulinskaya's lemma \cite[Lemma 11.2.10]{at07}, for fixed $u \in \R$ and a fixed box $B \subset \R^d$, almost surely the level set $\{f = u\}$ consists of smooth simple curves which intersects the boundary of $B$ transversally. A \textit{topological event} is an event that depends only on the stratified diffeomorphism class of $\{f \ge u\}|_{B}$ for some $u \in \R$. Examples are (i) the `crossing event' that $\{f  \ge u\}|_{B}$ contains a path that intersects two opposite $(d-1)$-dimensional faces, and (ii) the event that the number of connected components of $\{f  \ge u\}|_{B}$ exceeds a given threshold.

\begin{proposition}
\label{p:top}
Suppose there exists a $\delta > 0$ such that
\begin{equation}
\label{e:nondegen}
  \inf_{x \in \R^d}  \textrm{DetCov}[ (f(x), \nabla f(x) ) ]  > \delta   \quad \text{and} \quad  \sup_{x \in \R^d} \max \Big\{ \|  \E[  v(x) ]  \|_\infty, \| \Cov[ v(x) ] \|_{\infty} \Big\} < 1/\delta  ,
  \end{equation}
where $v(x)$ denotes the vector $(f(x), \nabla f(x), \nabla^2 f(x) ) \in \R \times \R^d \times \R^{d(d+1)/2}$. Then there exists a constant $c > 0$ depending only on $\delta$ and the dimension $d$ such that, for every box $B  \subset \R^d$, topological event $A \in \sigma(B)$, and $\eps > 0$,
\[       \P[f + \eps \in A] - \P[f \in A] \le c \eps  \overline{\textrm{Vol}}(B)  .\]
\end{proposition}

\begin{corollary}
\label{c:top}
Suppose there exists a $\delta > 0$ such that \eqref{e:nondegen} is satisfied. Then there exists a constant $c > 0$ depending only on $\delta$ and the dimension $d$ such that, for disjoint boxes $B_1,B_2 \subset \R^d$, and increasing topological events $A_1 \in \sigma(B_1)$ and $A_2 \in \sigma(B_2)$, 
\[ \Big| \P[A_1 \cap A_2] - \P[A_1]\P[A_2]  \Big| \le  c  \Big(  \overline{\textrm{Vol}}(B_1)  \overline{\textrm{Vol}}(B_2)    \|K_{B_1,B_2}\|_\infty \Big)^{1/3}   . \]
\end{corollary}
\begin{proof}
This is identical to the proof of Corollary \ref{c:cap}, using Proposition \ref{p:top} in place of Proposition \ref{p:cap}.
\end{proof}

\begin{remark}
Unlike Corollary \ref{c:cap}, Corollary \ref{c:top} continues to be effective even if the capacity is infinite, e.g.\ the monochromatic random waves in Example \ref{e:mrw}.
\end{remark}

\begin{remark}
The restriction to box domains is mainly for simplicity; in general the constants in Proposition \ref{p:top} and Corollary \ref{c:top} also depend on the maximum curvature of the boundary.
\end{remark}

\begin{example}
Suppose $f$ is a stationary $C^2$-smooth Gaussian field on $\R^d$ with $(f(0),\nabla f(0))$ non-degenerate and covariance satisfying $K(0,x) \|x\|_2^{2d} \to 0$ as $\|x\|_2 \to \infty$.  Then if $B_1 , B_2 \subset \R^d$ are disjoint boxes, by Corollary \ref{c:top} we have the quasi-independence estimate
 \[   \lim_{R \to \infty}  \sup_{ \stackrel{A_1 \in \sigma(RB_1), A_2 \in \sigma(R B_2)}{A_i \text{ increasing and topological} } }  \Big| \P[A_1 \cap A_2] - \P[A_1]\P[A_2]  \Big| = 0  .\]
\end{example}

In \cite{bmr20} an exact formula was derived for the covariance of topological events, not necessarily increasing, which implies a version of Corollary \ref{c:top} under slightly more restrictive assumptions; see \cite[Corollary 1.2]{bmr20}. In fact the conclusion of \cite[Corollary 1.2]{bmr20} is stronger and applies to all topological events, but its proof is more involved.

\subsection{Proof of the stability estimates}
\label{s:stabproof}

We finish the section with the proof of the stability estimates in Propositions \ref{p:cap} and \ref{p:top}.

\subsubsection{Proof of Proposition \ref{p:cap}}

Let $X$ and $Y$ be random variables on a common measurable space $\mathcal{X}$ with respective laws $\P$ and $\mathbb{Q}$. The \textit{relative entropy}, or \textit{Kullback-Leibler divergence}, from $Y$ to $X$ is defined as
\[ D_{KL}(X \| Y) = \int_\mathcal{X} \log \Big( \frac{\P(dx) }{\mathbb{Q}(dx)} \Big) \P(dx) . \]
The \textit{total variation distance} between $X$ and $Y$ is defined as
\[ d_{TV}(X, Y) = \sup_{\text{event } A} |\P(A) - \mathbb{Q}(A)|  . \]
These are related by Pinsker's inequality
\[ d_{TV}(X, Y)  \le \sqrt{ \frac{1}{2} D_{KL}(X \| Y)} . \]

Let $X$ be a Gaussian vector with covariance $K(i,j)$. The \textit{reproducing kernel Hilbert space (RKHS)} $H$ of $X$ is defined as the linear span of $(K(i, \cdot))_i$ equipped with the inner product
\[  \Big\langle  \sum a_i K(i, \cdot) ,  \sum b_i K(j, \cdot) \Big\rangle_H = \sum a_i b_i K(i, j ) . \]
A consequence of the Cameron-Martin formula is that, for every $h \in H$,
\begin{equation}
\label{e:cm}
 D_{KL}(X \| X+h ) = \frac{\|h\|_H^2 }{2}  .
 \end{equation}
In the setting of continuous Gaussian fields on $D \subseteq \R^d$ the RKHS is defined as the \text{closure} of the linear span of $(K(x_i, \cdot))_i$ under the same inner product, and \eqref{e:cm} remains true.

\begin{proof}[Proof of Proposition \ref{p:cap}]
Let $h \in H$ be such that $h|_I \ge 1$. Then since $A$ is increasing
\begin{align*}
     \P[X + \eps \in A] - \P[X \in A]  & \le \P[X + \eps h \in A] - \P[A]  \le d_{TV}(X,X + \eps h) \\
     & \le \sqrt{ \frac{1}{2} D_{KL}( X \| X + \eps h) } =  \frac{\eps \|h\|_H}{2} .   
     \end{align*}
This proves the inequality, since by the dual representation of the capacity
\begin{equation*}
\Capa(I) = \inf\{ \|h\|_H^2  :  h|_I \ge 1 \}  . \qedhere
\end{equation*}
\end{proof}

\subsubsection{Proof of Proposition \ref{p:top}}

Recall that $B \subset \R^d$ is a box equipped with its canonical stratification $\mathcal{S} = (S_i)_i$, and recall that $\overline{\textrm{Vol}}(B) = \sum_{S_i \in \mathcal{S}} \textrm{Vol}(S_i)$. Let $\textrm{vert}(B)$ denote the set of~$2^d$ vertices of~$B$.  For a point $x \in B \setminus \textrm{vert}(B)$, let $S(x)$ denote the stratum of $B$ that contains $x$, and let $\nabla|_{S(x)}$ denote the gradient operator restricted to $S(x)$. 

\smallskip
Let $U$ be an open neighbourhood of $B$, and let $g \in C^2(U)$. A point $x \in B$ is a \textit{stratified critical point} of $g$ if either $x \in \textrm{vert}(B)$ or if $x \in B \setminus \textrm{vert}(B)$ and $\nabla|_{B(x)} f(x) = 0$; its \textit{critical value} is $g(x)$. For levels $u < v$, let $N_B(g;u,v)$ denote the number of stratified critical points in $B$ with critical value in $[u,v]$.

\smallskip
The following is a basic lemma of stratified Morse theory (see \cite[Theorem 7]{han02}):

\begin{lemma}
\label{l:morse}
If $N_B(g; u,v) = 0$ then $\{g \ge u\}|_B$ and $\{g \ge v\}|_B$ are in the same stratified diffeomorphism class.
\end{lemma}

To estimate the probability that $N_B(f; u,v) = 0$ we bound its expectation:

\begin{lemma}
\label{l:kr}
Suppose there exists a $\delta > 0$ such that \eqref{e:nondegen} holds. Then there exists $c > 0$ depending only on $\delta$ and the dimension $d$ such that, for every box $B  \subset \R^d$ and levels $u < v$, 
\[ \E[ N_B(f; u, v) ]  \le  c  (v-u)  \overline{\textrm{Vol}}(B) . \]
\end{lemma}
 
\begin{proof}
Let $S \in \mathcal{S}$ be a stratum, and let $N_S$ denote the number of stratified critical points in $S$ with critical level in $[u,v]$. There are two cases:
 \begin{enumerate}
 \item $S$ is a vertex $\{x\}$. Then
\[ \E[ N_S ] = \P[ f(x) \in [u, v] ] =  \int_u^v \frac{1}{\sqrt{2d \sigma^2}}  e^{ -(\ell - \E[f(x)])^2 / (2 \sigma^2) } \, d\ell \le \frac{v-u}{\sqrt{2d \sigma^2}}  ,  \]
where $\sigma^2 = \Var[f(x)]$. 

\item $S$ is not a vertex. For $x \in S$, abbreviate $v^x = (f(x), \nabla|_{S} f(x) )$ and $N^x = (N_{i,j})_{i,j} =\nabla^2|_{S}f(x) $, and let $n = \textrm{dim}(S)$. Then by the Kac-Rice formula \cite[Corollary 11.2.2]{at07}
\begin{align*}
      \E[N_S] &=  \int_{x \in S, \ell \in [u,v]   } \varphi(\ell, 0) \E \big[ | \textrm{Det}(  N^x)  | \, \big| \, v^x = (\ell,  0 )  \big] \, dx d\ell  \\
      &  \le     (v-u) \textrm{Vol}(S) \sup_{x \in S, s \in \R^{d+1} }  \varphi(s)      \E \big[ | \textrm{Det}(N^x )  | \, \big| \,  v^x  = s  \big]   \\
         & \le   d^d  (v-u) \textrm{Vol}(S) \sup_{x \in S, s \in \R^{d+1}   }  \varphi(s)  \max_{i,j}   \E \big[ | N^x_{i,j} |^n   | \, \big| \,  v^x  = s  \big]   ,
      \end{align*}
      where $\varphi(s)$ is the density of $v^x$ at $s \in \R^{d+1}$, and in the final step we expanded the determinant and applied H\"{o}lder's inequality. 
\end{enumerate}
Applying Lemma \ref{l:gv} below, in both cases we have $\E[N_S] \le c (v-u) \textrm{Vol}(S)$ for a $c>0$ depending only on $\delta$ and $d$, which gives the result.
\end{proof}

\begin{proof}[Proof of Proposition \ref{p:top}]
Let $A$ be a topological event that depends on $\{f \ge u\}$. By Lemmas \ref{l:morse} and \ref{l:kr} we have 
\begin{equation*}
    \P[f + \eps \in A] - \P[f \in A] \le \P[ N_B(f; u-\eps, u)  \ge 1 ]  \le \E[ N_B(f; u-\eps, u) ]  \le c \eps \overline{\textrm{Vol}}(B)   . \qedhere
\end{equation*}
\end{proof}

In the proof of Lemma \ref{l:kr} we used the following property of Gaussian vectors:

\begin{lemma}
\label{l:gv}
Let $(X,Y)$ be an $(1 \times m)$-dimensional Gaussian vector and let $n \in \N$. Suppose there exists $\delta > 0$ such that
\[    \textrm{DetCov}[Y]   > \delta \quad \text{and} \quad  \max \Big\{  \| \E[(X,Y)] \|_\infty, \| \textrm{Cov}[ (X,Y) ] \|_\infty \Big\} < 1/\delta ,  \]
 and let $\varphi(s)$ denote the density of $Y$ at $s \in \R^m$. Then there exists a constant $\delta' > 0$ depending only on $m,n$ and $\delta$, such that,
\[ \min_{ I \subseteq \{1,\ldots,m\} } \textrm{Det} \textrm{Cov}[Y|_I]  > \delta'  \quad \text{and} \quad \sup_{s \in \R^m} \varphi(s) \E \big[ |X|^n \, \big| \, Y = s \big] < 1/\delta' . \]
\end{lemma}
\begin{proof}
Suppose $I \subset \{1,\ldots,m\}$ (the case $I = \{1,\ldots,m\}$ is trivial). Then
\[   \textrm{Det} \textrm{Cov}[Y|_I]  =     \textrm{Det} \textrm{Cov}[Y] /  \textrm{Det} \textrm{Cov}\big[ Y|_{I^c} \big| Y|_I \big] \ge     \textrm{Det} \textrm{Cov}[Y] /  \textrm{Det} \textrm{Cov}[Y|_{I^c} ] \ge m^{-m} \delta^{m+1}  \]
which gives the first item. For the second item, since $(X | Y=s)$ is Gaussian, we have
 \begin{align*}
 \E \big[ |X|^n \, \big| \, Y = s \big]  & \le c_n  \Big(  \big| \E \big[ X \, \big| \, Y = s \big] \big|^n  + \Var \big[ X \, \big| \, Y = s \big]   \Big) \\
 &   \le c_n \Big(  \big| \E \big[ X \, \big| \, Y = s \big] \big|^n  +  \Var [X] \Big)  .
 \end{align*}
 Moreover, diagonalising $\text{Cov}[Y]^{-1} = U^T \Lambda^{-1} U$ for orthogonal $U$ and diagonal $\Lambda = (\lambda_i)$, and by Gaussian regression,
\begin{align*}
  \sup_{s \in \R^m} \varphi(s)  \big| \E \big[ X \, \big| \, Y = s \big] \big|^n & =   \sup_{s \in \R^m} \varphi(s)  \Big| \E[X] +  \textrm{Cov}[X, Y]^T \textrm{Cov}[Y]^{-1} (s - \E[X] )  \Big|^n  \\
&  \le c_{m,\delta} \Big( 1 +  \sup_{s \in \R^m}  ( w^T \Lambda^{-1} s)^n  \, e^{ -\frac{1}{2}  s^T   \Lambda^{-1}   s }  \Big) 
\end{align*}
where $w = (w_i)_i =  U \textrm{Cov}[X, Y] $ and we made the substitution $s \mapsto U(s - \E[X])$. By explicit calculation one can check that
\[   \sup_{s \in \R^m}  ( w^T \Lambda^{-1} s)^n  \, e^{ -\frac{1}{2}  s^T   \Lambda^{-1}   s }   =  (n/e)^n \Big( \sum_i w_i^2 \lambda_i^{-1} \Big)^n \le c_n  \|w\|_\infty^{2n} \| \Lambda \|^{-n}_\infty \]
To finish, observe that by the Cauchy-Schwarz inequality, and since $\|U\|_\infty \le 1$,
\[ \|w\|_\infty \le c_m   \|\textrm{Cov}[(X,Y)]\|_\infty^{1/2}   \quad \text{and} \quad \|\Lambda\|_\infty^{-1} \le  c_m \| \textrm{Cov}[Y] \|^{m-1}_\infty  /  \textrm{DetCov}[Y]  .  \]
Gathering the estimates gives the result.
\end{proof}

\bigskip

\bigskip
\bibliographystyle{plain}
\bibliography{sdi}


\end{document}